%% file: paper.tex
\documentclass[12pt,reqno]{amsart}
\usepackage{amscd,amssymb,graphicx,color,a4wide,cite}
\usepackage{listings}
\usepackage{color}
\usepackage{booktabs}

\usepackage{hyperref}
\usepackage{url}

\definecolor{dkgreen}{rgb}{0,0.6,0}
\definecolor{gray}{rgb}{0.5,0.5,0.5}
\definecolor{mauve}{rgb}{0.58,0,0.82}

\usepackage{mathrsfs}

\usepackage{epstopdf}

\usepackage[all]{xy}
\newtheorem{theorem}{Theorem}[section]
\newtheorem{lemma}[theorem]{Lemma}
\newtheorem{proposition}[theorem]{Proposition}
\newtheorem{corollary}[theorem]{Corollary}

\theoremstyle{definition}
\newtheorem{definition}[theorem]{Definition}

\newtheorem{example}[theorem]{Example}

\theoremstyle{remark}
\newtheorem{remark}[theorem]{Remark}

\numberwithin{equation}{section}
\numberwithin{figure}{section}

\newcommand{\TT} {\mathbb{T}}
\newcommand{\ZZ} {\mathbb{Z}}

\newcommand{\RR} {\mathbb{R}}
\newcommand{\CC} {\mathbb{C}}

\newcommand{\PP} {\mathbb{P}}

\newcommand {\shB}  {\mathcal{B}}
\newcommand {\shC}  {\mathcal{C}}

\newcommand {\shF}  {\mathcal{F}}
\newcommand {\shG}  {\mathcal{G}}

\newcommand {\shI}  {\mathcal{I}}
\newcommand {\shJ}  {\mathcal{J}}

\newcommand {\shT}  {\mathcal{T}}

\newcommand {\shX}  {\mathcal{X}}

\newcommand {\foU}  {\mathfrak{U}}

\newcommand {\Aff}  {\operatorname{Aff}}

\newcommand{\cycleD} {\widetilde{D}}

\newcommand {\F} {\mathcal{F}}

\newcommand {\GL}  {\operatorname{GL}}

\newcommand {\Hom}  {\operatorname{Hom}}

\newcommand {\lra}  {\longrightarrow}

\newcommand {\M} {\mathcal{M}}

\renewcommand{\P}  {\mathscr{P}}

\newcommand {\Spec} {\operatorname{Spec}}

\newcommand {\Square}  {\operatorname{Sq}}

\renewcommand{\tilde}{\widetilde}

\def\mydate{\ifcase\month \or January\or February\or March\or
April\or May\or June\or July\or August\or September\or October\or 
November\or December\fi \space\number\day,\space\number\year}

\graphicspath{{images/}}

\newcommand\restr[2]{{% we make the whole thing an ordinary symbol
  \left.\kern-\nulldelimiterspace % automatically resize the bar with \right
  #1 % the function
  \vphantom{\big|} % pretend it's a little taller at normal size
  \right|_{#2} % this is the delimiter
  }}

\begin{document}

\title[Real Lagrangian]{Real Lagrangians in Calabi--Yau Threefolds}

\author{H\"ulya Arg\"uz}
\address{Laboratoire de Math\'ematiques de Versailles\\UVSQ--Paris Saclay\\55 Avenue de Paris, 78000 \\ France}
\email{nuromur-hulya.arguz@uvsq.fr}

\author{Thomas Prince}
\address{Mathematical Institute\\University of Oxford\\Woodstock Road\\Oxford\\OX$2$ $6$GG\\UK}
\email{thomas.prince@magd.ox.ac.uk}

\date{\today}

\begin{abstract}
	We compute the mod $2$ cohomology groups of real Lagrangians in torus fibrations on Calabi--Yau threefolds constructed by Gross. To do this we study a long exact sequence introduced by Casta\~{n}o-Bernard--Matessi, which relates the cohomology of the Lagrangians to the cohomology of the Calabi--Yau. We show that the connecting homomorphism in this sequence is given by the map squaring divisor classes in the mirror Calabi--Yau.
\end{abstract}
\maketitle

\tableofcontents
%\bigskip

% Input source files for each section
\input{introduction}

\input{integral_affine_manifolds}
\input{real_locus}
\input{main_proof}
\appendix
\input{lagrangian_topology}
\input{kato_nakayama}
%\input{EquivariantLocalization}

%-------------------------------------------------------------------------------
\bibliographystyle{plain}
\bibliography{bibliography}
%-------------------------------------------------------------------------------

\end{document}

%% file: introduction.tex
% !TEX root = paper.tex
%----------------------------------------------------------------------
\section{Introduction}
\label{sec:intro}
%----------------------------------------------------------------------

The celebrated Strominger--Yau--Zaslow (SYZ) conjecture \cite{SYZ} relates the complex geometry of a Calabi--Yau threefold with \emph{integral affine geometry} on a three-dimensional real manifold, the base of a conjectural special Lagrangian torus fibration. Although the construction of special Lagrangian submanifolds is a notoriously difficult problem of geometric analysis, this perspective has nonetheless given rise to substantial progress in understanding  mirror symmetry.

In one direction, one can weaken the conditions on the torus fibration: Ruan~\cite{Ruan,Ruan2,Ruan3}, Gross~\cite{GrossSLagI,GrossTopology,GrossBB}, and Haase--Zharkov~\cite{HZci} have constructed integral affine manifolds together with topological or Lagrangian torus fibrations on Calabi--Yau toric complete intersections, compatible with the mirror pairs constructed by Batyrev and Borisov \cite{B,BB}. In another direction, one can consider \emph{toric degenerations} of Calabi--Yau varieties, introduced by Gross--Siebert\cite{GS1}, and obtain torus fibrations on the general fibre, parametrised by choices of \emph{gluing data}, see \cite[\S$4$]{AS}. Such torus fibrations carry a canonical anti-holomorphic involution, which is anti-symplectic in the projective case, as discussed in Appendix~\ref{sec:Kato-Nakayama}, where we also review the study of fixed loci of such involutions carried out in the work of Siebert and the first named author~\cite{AS}.

In this article, we restrict our attention to torus fibrations on topological Calabi--Yau compactifications introduced by Gross~\cite{GrossTopology,GrossSLagI}. Here the base $B$ carries the structure of an integral affine manifold with simple singularities, as discussed in \S\ref{sec:integral_affine_manifolds}, and is homeomorphic to a $\ZZ_2$-homology sphere. For us, the most important condition that these fibrations satisfy is \emph{$G$-simplicity} for an abelian group $G$; if $f\colon X \to B$ is the fibration, and $f_0$ is the restriction of $f$ to the complement of the discriminant locus in $B$, then $G$-simplicity is statement that
\[
\iota_\star R^p {f_0}_\star G = R^p f_\star G 
\]
for all $p$ \cite[Definition $2.1$]{GrossSLagI}. Gross \cite{GrossTopology} has shown that three-dimensional $G$-simple torus fibrations, for appropriate $G$ (for instance $G = \ZZ$,  or $G = \ZZ_n$), admit canonical dual fibrations which are also $G$-simple. This is compatible with expectations from SYZ mirror symmetry. Among the topological Calabi--Yau compactifications that admit $G$-simple torus fibrations is one of most extensively studied examples in the literature: the quintic threefold $X \subset \PP^4$ \cite[Theorem~$0.2$]{GrossTopology}.

Casta\~no-Bernard--Matessi \cite{CBM3} have constructed Lagrangian torus fibrations on Calabi--Yau compactifications. In their setting there is a canonical anti-symplectic involution on the total space $X$ of the compactification \cite{CBM, CBMS}; the fixed point locus of this involution is a real Lagrangian submanifold $L_\RR \subset X$. There is an extensive study of real Lagrangians in the literature, as they provide an algebro-geometric path to open Gromov--Witten invariants and the Fukaya category. For previous work in this direction, see \cite{FOOO, Solomon, PSW, Penka}. 

Given a Lagrangian torus fibration $f \colon X \to B$ as in \cite{CBM3}, let $\pi \colon L_\RR \to B$ denote the restriction of $f$ to the real Lagrangian $L_\RR$. As observed in \cite{CBM}, there is a short exact sequence of sheaves
\[
\label{eq:ses_intro}
0 \to R^1f_\star\ZZ_2\oplus\ZZ^2_2 \to \pi_\star\ZZ_2 \to R^2f_\star\ZZ_2 \to 0
\]
and the $\ZZ_2$-cohomology of $L_\RR$ is determined by the connecting homomorphism
\[
\beta \colon H^1(B, R^2f_\star\ZZ_2) \to H^2(B, R^1f_\star\ZZ_2)
\]
in the associated long exact sequence. The computation of $\beta$ was stated as an open problem in \cite[\S$1$]{CBM}. In this article we resolve this problem, and give an explicit description of $\beta$ in terms of the pairing on the cohomology of the mirror $\breve{X}$. Our main result is the following, which appears in the article as Theorem~\ref{thm:beta_is_square}.

\begin{theorem}
	The connecting homomorphism $\beta$ in the long exact sequence \eqref{eq:les_CBM} coincides with the map
\begin{eqnarray}
\nonumber
\Square \colon H^1(B,R^1\breve{f}_\star\ZZ_2) & \lra & H^2(B,R^2\breve{f}_\star\ZZ_2) \\
\nonumber
D & \longmapsto & D^2.
\end{eqnarray}	
\end{theorem}

If $H^1(X,\ZZ_2) = 0$, the map $\Square$ in the above theorem coincides with the usual cup product in cohomology. As a consequence, we obtain the following result as Corollary~\ref{cor:mod2_cohomology_again}. 
\begin{corollary}
	If $H^1(\breve{X},\ZZ_2) = 0$, then
	\[
	h^1(L_\RR,\ZZ_2) = h^1(B,R^1f_\star\ZZ_2) + \dim\ker(\Square ).
	\]
	Moreover if, in addition, $H^2(\breve{X},\ZZ) \cong \ZZ$, $H^3(\breve{X},\ZZ)$ contains no $2$-torsion, and $H^1(X,\ZZ_2) = 0$, then
	\begin{equation}
	\nonumber
h^1(L_\RR,\ZZ_2) = h^1(B,R^1f_\star\ZZ_2) + \delta.
\end{equation}
Here 
\[
\delta = 
\begin{cases}
1 & \text{if $\overline{D}^3$ is divisible by $2$} \\
0 & \text{otherwise,}
\end{cases}
\]
where $\overline{D}$ is a generator of $H^2(\breve{X},\ZZ)$.
\end{corollary}

Note that the rank of the map $\Square\colon D \mapsto D^2$ can be easily computed whenever the intersection form on $H^2(X,\ZZ)$ is known. We illustrate this computation for the quintic threefold in Example~\ref{eg:quintic}, following \cite[\S$4$]{GrossTopology}. Moreover, focusing on the mirror $\breve{X}$ to the quintic threefold, from Corollary~\ref{cor:mod2_cohomology_again} and Remark~\ref{Rem: number of connected components} we deduce that real Lagrangian $\breve{L}_\RR \subset \breve{X}$ has the following Betti numbers
\begin{eqnarray}
\nonumber
h^0(\breve{L}_\RR,\ZZ_2) & = & h^3(\breve{L}_\RR,\ZZ_2) = 2 \\
\nonumber
h^1(\breve{L}_\RR,\ZZ_2) & = & h^2(\breve{L}_\RR,\ZZ_2) = 101.
\nonumber
\end{eqnarray}

Beyond the case of the quintic threefold and its mirror, $\ZZ$-simple Calabi--Yau compactifications exist on a vast number of threefolds, for example, on all Batyrev--Borisov mirror pairs \cite{CBM2, GrossBB}. Given a Calabi--Yau complete intersection $X$ in a toric variety, the dual $\breve{f} \colon \breve{X} \to B$ is expected to be diffeomorphic to a specific non-singular minimal model of the mirror. This is proven for the quintic threefold in \cite[Theorem~$0.2$]{GrossTopology}, and is expected to follow analogously for complete intersections in toric varieties, see~\cite{GrossBB}.

In general, it is expected that different mirrors to a given Calabi--Yau compactification can be realised geometrically by passing between different \emph{large complex structure limit points}. This combinatorially corresponds to applying \emph{flips}, or with the slightly different terminology used in \cite[\S$4$]{GrossTopology} \emph{toric flops}, on the base $B$. In Appendix~\ref{sec:flips}, we describe how the topology of real Lagrangians in mirrors to Calabi--Yau compactifications changes under flips, and prove that performing a flip on $B$ induces a Dehn surgery on $\breve{L}_\RR$. We obtain the following as Theorem~\ref{Invariance under flips}.

\begin{theorem}
	The dimension of $H^1(\breve{L}_\RR,\ZZ_2)$ remains invariant under flipping the base $B$.
\end{theorem}

In certain examples we find that the ranks of the mod $2$ cohomology groups of the real Lagrangians we study are related to the Hodge numbers of the corresponding Calabi--Yau varieties, as remarked in \cite[p.$37$]{CBMS}. This relationship has also been studied for toric varieties in \cite{How08}. Investigating further in this direction is focus of future work \cite{AP2}.

\subsection*{Related works} The mod $2$ cohomology groups of fixed point sets of anti-symplectic involutions on symplectic manifolds with Hamiltonian torus actions can be computed using Equivariant Localisation Theorems \cite{hsiang2012cohomology,borel2016seminar,atiyahequivariant,quillen1971spectrum}. For previous work in this direction we refer to \cite{biss2004mod2}. 

The construction problem of algebraic varieties with real structures is studied in \cite{AS}, using Kato--Nakayama spaces, as summarised in Appendix~\ref{sec:Kato-Nakayama}. Another very influential method of constructing algebraic varieties with real structures combinatorially is Viro's patchworking method \cite{Viro}. This method produces real algebraic hypersurfaces inside toric varieties. Bounds on the mod $2$ Betti numbers of the real locus for such varieties have been studied by Bihan, Itenberg, and Viro \cite{itenbergOne,itenbergTwo,Bihan} and, very recently, by Renaudineau--Shaw~\cite{ArthurKirstin} via the introduction of a real analog of tropical homology groups \cite{TropicalHomology}. Our context is simultaneously more and less general than this. The varieties we consider in this paper typically are not toric hypersurfaces, but are restricted to the three-dimensional Calabi--Yau case. We expect our results to agree with \cite{ArthurKirstin} for the class of three-dimensional toric Calabi--Yau hypersurfaces.

------------------------------------------------------------------
\begin{flushleft}
\textbf{Acknowledgements.}
\end{flushleft}
%----------------------------------------------------------------------
We thank Mark Gross, Bernd Siebert, and Tom Coates for many useful conversations. We also thank Paul Seidel for suggesting the question and for his encouragement throughout. Finally, we thank the anonymous referee for their many insightful comments and valuable suggestions which have resulted in major improvements to this article. This project has received funding from the European Research Council (ERC) under the European Union's Horizon $2020$ research and innovation programme (grant agreement No. 682603). TP was partially supported by a Fellowship by Examination at Magdalen College, Oxford.

%% file: integral_affine_manifolds.tex
% !TEX root = paper.tex
%----------------------------------------------------------------------

\section{From real affine to symplectic geometry}
\label{sec:integral_affine_manifolds}

Throughout the paper $M$ denotes a free abelian group of rank $n\geq 2$, and $M_{\RR} = M\otimes_{\ZZ}\RR$ is the associated real vector space.
\begin{definition}
\label{Def: topological torus}
Let $X$ and $B$ be topological manifolds with $\dim X = 2n$ and $\dim B = n$. A \emph{topological $\mathbb{T}^n$ fibration} on $X$ is a continuous, proper map $f \colon X \to B$ with connected fibres, such that 
the fibre $X_b = f^{-1}(b)$
is homeomorphic to an $n$-torus $\mathbb{T}^n$ for all $b \in B_0$, for an open dense set $B_0 \subseteq B$. The set 
\[
\Delta :=  B \setminus B_0
\]
is called \emph{the discriminant locus} of $f$. 
\end{definition}
When $X$ is a symplectic manifold, with symplectic form $\omega$, a topological $\mathbb{T}^n$-fibration on $X$ is called \emph{Lagrangian} if the restriction of $\omega$ to the smooth part of every fibre vanishes. When studying Lagrangian fibrations on spaces with interesting topology, such as Calabi--Yau manifolds other than complex tori, the discriminant locus is almost invariably non-empty. In this case the base $B$ carries the structure of an affine manifold with singularities, which we define following the terminology in \cite[Definition~$2.1$]{GS4}, as follows.
\begin{definition}
\label{Def:singular affine structure}
An \emph{affine manifold} is a topological manifold together with an open cover, for which the change of coordinate transformations are elements of the group of affine transformations of $M_\RR$,  
\[
\Aff(M_\RR) := M_\RR \rtimes GL_n(\RR).
\] 
Such a manifold is called an \emph{integral} affine manifold if the change of coordinate transformations lie in 
\[
\Aff(M) := M \rtimes GL_n(\ZZ).
\] 
An (integral) \emph{affine manifold with singularities} is a topological manifold $B$ which admits an (integral) affine structure on a subset $B_0 := B\setminus \Delta$, where $\Delta \subset B$ is a union of submanifolds of $B$ of codimension at least $2$. The union of these submanifolds is called the {\it discriminant locus} of $B$. 
\end{definition}
Given a Lagrangian fibration $f\colon X \to B$, there is a natural affine structure on $B$, for which the discriminant locus agrees with the discriminant locus of $f$ \cite{D80, cushman, zung}. Moreover, this affine structure is integral if and only if the symplectic form on $X$ represents an integral cohomology class by \cite[Remark~$5.10$]{sepe2013universal} -- see also \cite[Remark~$1.2$]{GS1}. Conversely, fixing an affine manifold with singularities, describing the obstructions to constructing a torus fibration over it is a technically challenging problem, and is discussed further in Appendix~\ref{sec:Kato-Nakayama}.

In this paper, following work of Gross \cite{GrossTopology}, we study torus fibrations on three-dimensional Calabi--Yau compactifications, constructed from integral affine manifolds with \emph{simple} singularities, as defined in \cite[Definition~$3.14$]{CBM3} and \cite[Definition~$6.95$]{DBranes09}. An analogous notion of simplicity was introduced in \cite[Definition~$1.60$]{GS1}, as an indecomposability condition on the local affine monodromy around the discriminant locus of the affine manifold. In dimension three this locus is a trivalent graph with \emph{positive} and \emph{negative} vertices, as discussed further in \S3; see Appendix~\ref{sec:flips} for explicit affine monodromy computations. The construction of topological Calabi--Yau compactifications over integral affine manifolds with simple singularities is explained in \cite{GrossTopology}, and using slightly different notation, in \cite[Chapter~$6$]{DBranes09}. In particular, it is shown how to form a (singular) fibration $f \colon X \to B$ fitting into the commutative diagram
\begin{equation}
\nonumber
	\xymatrix{
		X_0 \ar[d]^{f_0} \ar@{^{(}->}[r] & X \ar[d]^f \\
		B_0 \ar@{^{(}->}[r]^\iota & B, \\
	}
	\label{compactification}
	\end{equation}
such that $f$ is a `good' compactification of $f_0 \colon X_0 := T^\star {B_0}/\breve{\Lambda} \to B_0$, where $\breve{\Lambda} \subset T^\star B_0$ is the canonical covariant lattice of integral cotangent vectors. In general it is not clear which compactifications $f \colon X \to B$ should be considered `good', or how to dualize such a compactification. However, in the case when $B$ has simple singularities, such compactifications and their duals are obtained in a reasonably unique way \cite{GrossTopology}. It was later shown in \cite{CBM2} that this procedure can be carried out in the symplectic category. In particular, after a suitable deformation, the topological fibrations of \cite{GrossTopology} can be made into piecewise smooth Lagrangian fibrations. The local monodromy groups of the singular fibres of these fibrations are unipotent, and a classification of the possible monodromy groups is given in \cite[\S2]{GrossTopology}. 

A related notion to the simplicity of the base $B$, is the notion of \emph{$G$-simplicity} of a $\mathbb{T}^n$-fibration over $B$, defined as follows \cite[Definition~$2.1$]{GrossSLagI}.  
\begin{definition}
\label{Gsimple}
A $\mathbb{T}^n$-fibration $f\colon X\rightarrow B$ is \emph{$G$-simple}, for an abelian group $G$, if  
\[
i_\star R^p{f_0}_\star G =R^pf_\star G
\]
for all $p$, where $R^p{f_0}_\star G$, respectively $R^pf_\star G$, is the sheaf associated to the presheaf on $B_0$, respectively on $B$, given by $U \mapsto H^k(f_0^{-1}(U),G)$, respectively $U \mapsto H^k (f^{-1}(U),G)$.
\end{definition}

The condition of $G$-simplicity implies that the cohomology of the singular fibres is determined by the monodromy of the local system $R^q{f_0}_\star G$ on $B_0$, and is isomorphic to the monodromy invariant part of the cohomology of a nearby non-singular fibre. This condition is crucial in relating the cohomology of $X$ with that of the dual fibration. It is shown in \cite{GrossTopology} that the Calabi--Yau compactifications obtained from integral affine manifolds with simple singularities are $G$-simple for $G=\ZZ$ and $G=\ZZ_n$.

%% file: real_locus.tex
% !TEX root = paper.tex
%----------------------------------------------------------------------
\section{The real Lagrangian as a multi-section}
\label{sec:real_locus}
%----------------------------------------------------------------------

We are now in a position to introduce our main actors. This section has considerable overlap with \cite[\S$2$]{CBM}. Throughout this article we let $B$ be an integral affine manifold with simple singularities, as defined in \cite[Definition~$6.95$]{DBranes09} and \cite[Definition~$3.14$]{CBM3}. In particular, we recall that $\Delta$ is a trivalent graph, with a partition on its set of vertices into positive and negative vertices. Recall that $B_0 = B \setminus \Delta$, and there is a $\ZZ$-simple topological Calabi--Yau compactification $f \colon X \to B$ of the fibration $f_0\colon X_0 \to B_0$, using the local models described in \cite{CBM3,GrossTopology,DBranes09}. While we do not recall the full definition of the map $f$, we briefly recall a description of the fibres used in the compactification.

\begin{enumerate}
\label{positive-negative}
	\item If $p \in \Delta$ is a point which is not a vertex of $\Delta$, then $f^{-1}(b)$ is homeomorphic to the product of a pinched torus with $S^1$.
	\item If $p \in \Delta$ is a \emph{positive vertex}, then $f^{-1}(b)$ is homeomorphic to $S^1 \times S^1 \times S^1/ \sim$ where $(a, b, c) \sim (a', b', c')$ if $(a,b, c) = (a', b', c')$, or $a = a' = 1$. This is a three dimensional analogue of a pinched torus and $\chi(f^{-1}(b)) = 1$.
	\item If $p \in \Delta$ is a \emph{negative vertex}, then $f^{-1}(b)$ is homeomorphic to $S^1 \times S^1 \times S^1/ \sim$, where  $(a, b, c) \sim (a', b', c')$ if  $(a, b, c) = (a', b', c')$ or  $a = a' = 1, b = b'$, or $a = a', b = b' = 1$. The singular locus of this fibre is a figure eight, and $\chi(f^{-1}(b)) = -1$.	
\end{enumerate}

Note that the dual fibration of a $\ZZ$-simple fibration $f$ is given by a similar compactification $\breve{f} \colon \breve{X} \to B$ of the torus bundle $\breve{f}_0 \colon \breve{X}_0  \to B_0$, which exchanges the local models used around positive and negative vertices. We refer to \cite{GrossTopology} for further details.

We view each fibre of $f_0$, and of $\breve{f}_0$, as the quotient space $T^3 =\RR^3/\ZZ^3$. Throughout this paper we consider the canonical involution on each such fibre, given by taking $x \mapsto -x$. The fixed points of this involution can be represented as the eight half-integral points in the unit cube, as illustrated in Figure~\ref{fig:monodromy_cube}. This involution on $B_0$ extends over fibres $f^{-1}(p)$ for $p \in \Delta$, and hence defines an anti-symplectic involution $\iota$ on $X$ whose fixed locus is a real Lagrangian submanifold \cite{CBM}. For technical details on extending the fixed locus over the discriminant locus in a more general set up we refer to \cite[Remark~$4.16$, \S$4.3$]{AS}. We let
\[
\pi\colon L_\RR \to B
\] 
denote the restriction of $f$ to the fixed point locus of $\iota\colon X \to X$, and let $\pi_0$ denote the restriction of $\pi$ to $\pi^{-1}(B_0)$. Similarly, we let $\breve{L}_\RR$ denote the fixed point locus of the corresponding involution on the mirror fibration $\breve{f}\colon \breve{X} \to B$. We define the affine monodromy around a point on the discriminant locus on $B$ as follows.
\begin{definition}
\label{Def: affine monodromy}
Let $B$ be an affine manifold with singularities. Fix a point $p\in B\setminus B_0$. Let $\gamma\colon S^1\rightarrow B$ be a loop based at $p$ and let $U_1,\ldots,U_n$ be a collection of open sets covering the image of $\gamma$. Denote by $A_{i,i+1}^{-t}$ the inverse transpose of the linear part of the change of coordinate function defined on $U_{i}\cap U_{i+1}$. The \emph{affine monodromy representation} $\psi\colon \pi_1(B,p) \to \GL_n(\ZZ)$ is defined by setting
\[
\psi =
\left\{
	\begin{array}{ll}
		A_{1,n}^{-t} \cdots A_{2,1}^{-t}  & \mbox{if } n \geq 2 \\
		\mathrm{Id} & \mbox{otherwise.} 
	\end{array}
\right.
\]
\end{definition}

\begin{figure}
		\includegraphics*[scale=1.0]{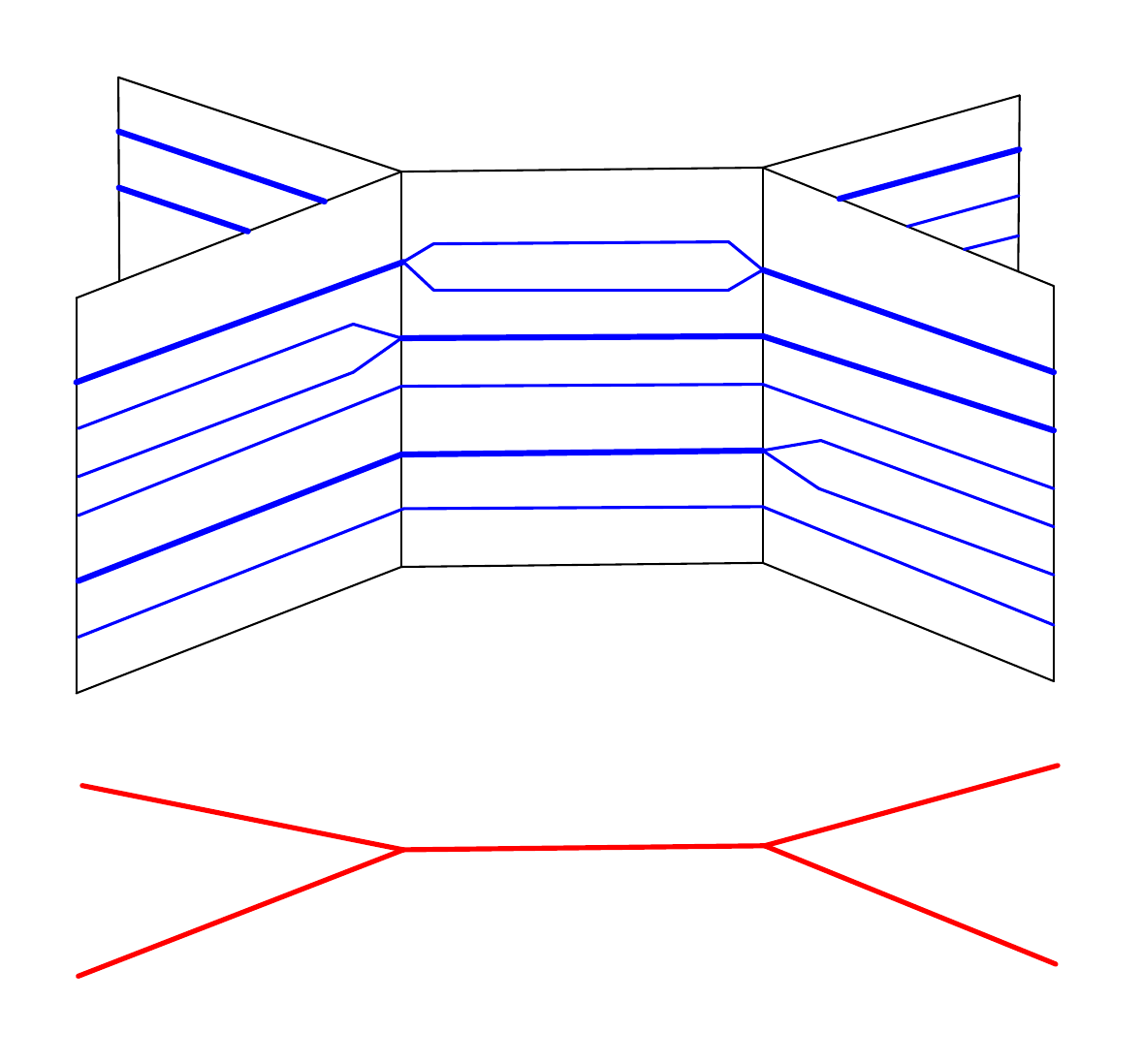}
		\caption{Cross sections of the real Lagrangian $L_\RR$ in the quintic threefold over a part of $\Delta$.}
		\label{fig:Sections}
\end{figure}

Note that the definition of affine monodromy is independent of the chosen representative $\gamma$ in the class $[\gamma] \in \pi_1(B,p)$ \cite{A,AM}. Moreover, the affine monodromy is, by definition, the inverse transpose of the linear part $T_\gamma$ of the standard monodromy representation around a loop $\gamma$ in $B_0$, see \cite[Definition~$1.4$]{GS1}. The points in $\pi^{-1}(p)$ exchanged under the monodromy action are illustrated in Figure~\ref{fig:monodromy_cube}, in which each point in $\pi^{-1}(p)$ corresponds to the vertex with label $i$. Figure~\ref{fig:monodromy_cube} displays, from left to right:
\begin{enumerate}
	\item The orbits of the $\ZZ_2$ action induced by monodromy around a single edge of $\Delta$.
	\item The orbits of the $\ZZ^3_2$ action induced by monodromy around three different edges of $\Delta$ adjacent to a positive vertex.
	\item The orbits of the $\ZZ^3_2$ action induced by monodromy around three different edges of $\Delta$ both adjacent to a negative vertex.
\end{enumerate} 
We describe the affine monodromy around loops in $B$ explicitly in Example~\ref{QuinticBase}, focusing on a torus fibration on the quintic threefold $X \subset \PP^4$, described in \cite[Theorem~$0.2$, \S$2$]{GrossTopology}. Given that the origin in Figure~\ref{fig:monodromy_cube} is fixed under every monodromy action, the following Lemma follows immediately. For an analogous result in greater generality, see also Theorem \ref{Thm: AS}.
\begin{lemma}
\label{lem:connected_cmpts}
Let $X\to B$ be a $T^3$ fibration as constructed as in \cite{GrossTopology, CBM2}.	The fixed point locus $L_\RR \subset X$ of the anti-symplectic involution $\iota$ on $X$, is a $2^3$-to-$1$ covering of $B$ branched along the discriminant locus $\Delta \subset B$. Moreover, at least one connected component of $L_\RR$ maps homeomorphically onto $B$.
\end{lemma}
The sections of the real Lagrangian $L_\RR \subset X$, where $X \subset \PP^4$ is the quintic threefold studied in Example~\ref{QuinticBase}, over a part of the discriminant locus lying in the interior of a two-dimensional face of $B$ is illustrated in Figure~\ref{fig:Sections}. 

\begin{remark}
We refer to \cite[\S$4.3$]{AS} for a result analogous to Lemma~\ref{lem:connected_cmpts} on Kato--Nakayama spaces. Recall that throughout this paper we restrict our attention to the compactified torus fibrations constructed in \cite{GrossTopology}, the case referred to as \emph{trivial gluing data} in \cite[Remark~$4.16$]{AS}. For a more detailed discussion on this see Appendix~\ref{sec:Kato-Nakayama}. Twisting the gluing data results in topologically distinct real Lagrangians, and the number of connected components can vary, see \cite[Example~$2.10$]{AS}. 
\end{remark}

\begin{figure}
	\makebox[\textwidth][c]{\includegraphics[scale=2]{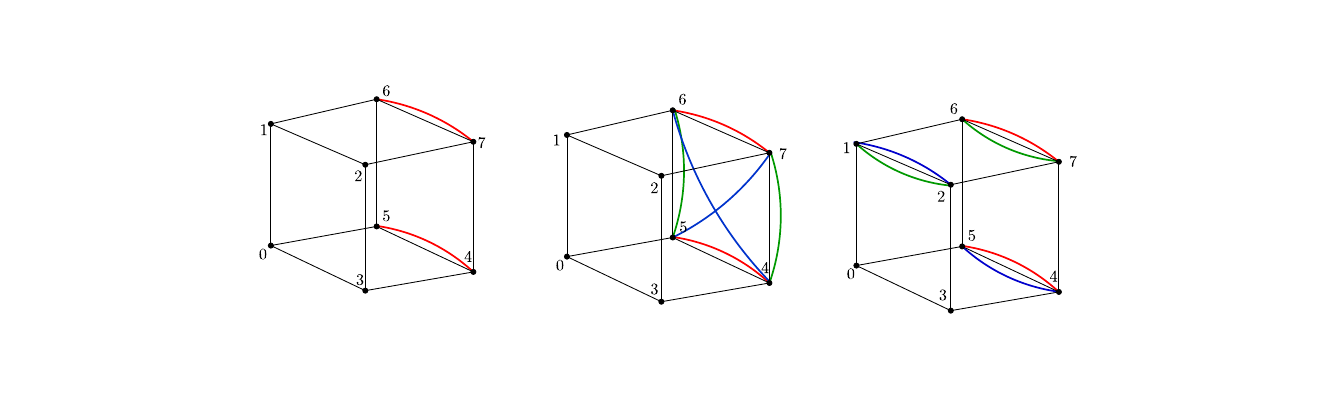}}
	\caption{The action of monodromy around a single branch, and the branches adjacent to a positive and negative vertex of $\Delta$, respectively.}
	\label{fig:monodromy_cube}
\end{figure}

%% file: main_proof.tex
% !TEX root = paper.tex
%----------------------------------------------------------------------
\section{Proof of the Main Theorem}
\label{sec:main_result}
%----------------------------------------------------------------------

\subsection{$\Square$ is equal to $\beta$}
\label{sec:beta_to_square}

Let $B$ be a three-dimensional integral affine manifold with simple singularities, and let
\[
f\colon X \to B
\]
denote the compactification of the $\TT^3$-fibration $f_0 \colon T^\star B/\breve{\Lambda} \to B_0$, as obtained in \cite{GrossTopology, CBM3}. Recall from Definition~\ref{Def:singular affine structure} that the fibration $f$ has singular fibres over the discriminant locus $\Delta \subset B$. Let 
\begin{equation}
\label{real involution}
\iota\colon X \to X
\end{equation}
be the involution whose restriction to smooth fibres is given by $x \mapsto -x$. Recall that we let $\pi$ denote the restriction of the map $f$ to the fixed point locus $L_\RR$ of $\iota$. The map $\pi$ is a $2^3$-to-$1$ branched covering of $B$, ramified over $\Delta$, as described in \S\ref{sec:real_locus}.

In \cite{CBM} the authors prove that there is a short exact sequence relating $\pi_\star L_\RR$ to the sheaves $R^1f_\star\ZZ_2$ and $R^2f_\star\ZZ_2$. We recall its construction in Proposition~\ref{pro:ses} below. The associated long exact sequence relates the mod $2$ cohomology groups of $L_\RR$ with the affine Hodge groups $H^j(B,R^if_\star\ZZ_2)$, for $i,j \in \{0,1,2,3\}$. An open question -- stated in the introduction of \cite{CBM} -- is to compute the connecting homomorphism
\[
\beta \colon H^1(B,R^2f_\star\ZZ_2) \to H^2(B,R^1f_\star\ZZ_2)
\]
induced by \eqref{eq:ses_1} in explicit examples, such as the quintic threefold, or complete intersections in toric manifolds. We give an answer to this question, applicable to any example, in Theorem~\ref{thm:beta_is_square}; allowing us to deduce the ranks of the mod $2$ cohomology groups of $L_\RR$.

Following \cite[\S$3$]{CBM}, we introduce the sheaves
\[  \shG   :=  R^2{f_0}_\star\ZZ_2, \,\ \,\ \,\ \,\  \shG'  :=  {\pi_0}_\star\ZZ_2, \,\ \,\ \,\ \,\  \shG^\vee  :=  R^1{f_0}_\star\ZZ_2. \]

\begin{remark}
	\label{rem:identifications}
	Note that, following \cite[p.$246$]{CBM}, we identify the set underlying the group $\shG_p \cong H^2(\TT^3,\ZZ_2) \cong \ZZ_2^3$ with the fibre $\pi^{-1}(p)$ for any $p \in B_0$. The group $\shG^\vee_p$ is the dual group to $\shG_p$, following the notation used in \cite{CBM}. 
\end{remark}

Before proceeding we fix the following notational convention. Given a set $V$ and a subset $U \subset V$, we let $1_U \colon V \to \ZZ_2$ denote the indicator function
\[
1_U(v) =
\begin{cases}
1 &  \textrm{ if } v \in U \\
0 & \textrm{ otherwise.}
\end{cases}
\]

\begin{proposition}[{\!\!\cite[\S$3$]{CBM}}]
\label{pro:ses}
There is a short exact sequence of sheaves
\begin{equation}
\label{eq:ses_1}
	0 \lra R^1f_\star\ZZ_2\oplus\ZZ^2_2 \lra \pi_\star\ZZ_2 \lra R^2f_\star\ZZ_2 \lra 0
	\end{equation}
\end{proposition}
\begin{proof}
We first define the short exact sequence \eqref{eq:ses_1} over the locus $B_0 = B\setminus \Delta$, recalling that $\Delta \subset B$ denotes the discriminant locus. In particular, we first describe a short exact sequence of sheaves
\begin{equation}
	\label{eq:ses_G}
	0 \lra \shG^\vee \oplus\shC \lra \shG' \lra \shG \lra 0,
\end{equation}
where $\shC$ is the constant sheaf $\ZZ^2_2$ on $B_0$. Recalling the identifications described in Remark~\ref{rem:identifications}, $\pi^{-1}(p)$ has a group structure isomorphic to $\shG_p \cong \ZZ_2^3$. The stalk $\shG'_p$ is nothing but the set of (set theoretic) maps from $\shG_p$ to $\ZZ_2$, while the set of linear maps from $\shG_p$ to $\ZZ_2$ is equal to $\shG^\vee_p$. Thus $\shG^\vee$ is naturally a subsheaf of $\shG'$. Indeed, the map $\shG^\vee \oplus\shC \to \shG'$ appearing in \eqref{eq:ses_G} is the inclusion of the following functions into $\shG'_p$:
\begin{enumerate}
	\item The set $\shG^\vee$ of linear maps $\ZZ_2^3 \to \ZZ_2$.
	\item The indicator function $1_{\{0\}}$, evaluating to $1$ at the origin and to $0$ elsewhere.
	\item The constant function $1_V$.
\end{enumerate}
The set of maps generated by $1_{\{0\}}$ and the constant function $1_V$ is denoted by
\[
\shC_p = \langle 1_{\{0 \}}, 1_V  \rangle \cong \ZZ_2^2.
\]
The map $\shG' \to \shG$ in \eqref{eq:ses_G} is defined as follows. Let $g$ be an element of the group $\shG_p$; one can show that every class in the quotient of $\shG'_p$ by $\shG^\vee_p \oplus\shC_p$ is represented by an element $1_{\{g\}}$ for a unique $g \in \shG_p$. Hence the map from $\shG'_p$ to $G_p$ sends every element in the class of $1_{\{g\}}$ to $g$. This map is linear, and extends to a morphism of sheaves. Using $\ZZ_2$-simplicity, we push forward the sheaves in the short exact sequence \eqref{eq:ses_G} by $j_\star$, where 
\[
j\colon B_0 \hookrightarrow B
\] 
denotes the canonical inclusion, to obtain the exact sequence \eqref{eq:ses_1}.
\end{proof}

We assume for the remainder of this section that $B$ is a $\ZZ_2$-homology sphere. In particular, we have that
\[
H^j(B,R^1f_\star\ZZ_2 \oplus \ZZ_2^2) = H^j(B,R^1f_\star\ZZ_2)
\]
for each $j \in \{1,2\}$. Moreover, in this setting, the long exact sequence associated to the short exact sequence described in Proposition~\ref{pro:ses} simplifies to
\begin{eqnarray}
\label{eq:les_CBM}
\nonumber
0 \lra H^0(B, R^1f_\star \ZZ_2) \oplus \ZZ^2_2  & \lra & H^0(B,\pi_\star \ZZ_2)  \to  H^0(B,R^2f_\star \ZZ_2)  \lra \\ 
H^1(B, R^1f_\star \ZZ_2) & \lra & H^1(B,\pi_\star \ZZ_2)  \lra  H^1(B,R^2f_\star \ZZ_2)  \stackrel{\beta} \lra \\
H^2(B, R^1f_\star \ZZ_2) & \lra & H^2(B,\pi_\star \ZZ_2)  \lra  H^2(B,R^2f_\star \ZZ_2)  \lra 0. 
\nonumber
\end{eqnarray}
The Leray spectral sequence associated to the function $\pi \colon L_\RR \to B$ has the form
\[
H^j(B, R^i\pi_\star\ZZ_2) \implies H^{i+j}(L_\RR, \ZZ_2).
\]
Noting that the sheaf $R^i\pi_\star\ZZ_2$ is trivial for all $i \in \ZZ_{> 0}$, we conclude that
\begin{equation}
\label{leray-on-pi}
H^i(B, \pi_\star\ZZ_2) \cong H^i(L_\RR, \ZZ_2)
\end{equation} 
for all $i\in \{ 0, 1, 2, 3\}$. To deduce the ranks of the cohomology groups of the real locus $L_\RR$ using \eqref{leray-on-pi}, it suffices to compute the connecting homomorphism $\beta$ in the sequence \eqref{eq:les_CBM}. 

In order to describe $\beta$ we make use of the following dualities between sheaves associated to the fibration $f$ and $\breve{f}$, following \cite[\S$2$]{GrossSLagI}. 

\begin{lemma}
	\label{lem:dual_sheaves}
	There are canonical identifications
	\begin{align*}
	\mu_1 \colon  C^1(B,R^1\breve{f}_\star \ZZ_2) &\to C^1(B,R^2f_\star \ZZ_2),\\
	\mu_2 \colon  C^2(B,R^2\breve{f}_\star \ZZ_2) &\to C^2(B,R^1f_\star \ZZ_2).
	\end{align*} 
Fixing a point $p \in B_0$, these identifications specialise to the usual isomorphisms 
\[
H^a((\TT^3)^\star,\ZZ_2) \cong H^{3-a}(\TT^3,\ZZ_2)
\]
of stalks at $p$, for $a \in \{1,2\}$.
\end{lemma}
\begin{proof}
	These identifications follow immediately from the $\ZZ_2$-simplicity of $f$ and $\breve{f}$, see \cite[$(2.1)$]{GrossSLagI}.
\end{proof}

The isomorphisms $\mu_1$ and $\mu_2$ in Lemma \ref{lem:dual_sheaves}, define isomorphisms on \u{C}ech cohomology groups, which we denote by
$\bar{\mu}_i \colon H^i(B,R^i\breve{f}_\star \ZZ_2) \to H^i(B,R^{3-i}f_\star \ZZ_2)$ for each $i \in \{1,2\}$.

We also observe that we can remove the constant sheaf factors appearing the sequence \eqref{eq:ses_1}, and throughout this section we work with the following sequence.
\begin{lemma}
\label{lem:ses_2}
There is a short exact sequence of sheaves 
\begin{equation}
\label{eq:ses_2}
0 \lra R^1f_\star\ZZ_2 \lra \shF  \lra R^2f_\star\ZZ_2  \lra 0,
\end{equation}
where $\shF := \pi_\star\ZZ_2/\langle 1_{\{0\}},1_V \rangle$.
\end{lemma}
\begin{proof}
Fixing a point $p \in B_0$, the inclusion $\shC=\ZZ_2^2 \to \shG'$, defined in the proof of Proposition~\ref{pro:ses}, splits via the projection $\shG' \to \ZZ_2^2$ defined by
\[
g \mapsto \big(\sum_{i \in \ZZ_2^3}g(i), \sum_{i \in \ZZ_2^3\setminus \{0\}}g(i)\big),
\]
for each $g \in \shG'_p$. Pushing forward along the inclusion $j \colon B_0 \to B$, and noting that $\pi_\star \ZZ_2 = j_\star \shG'$, we obtain a splitting
\[
\pi_\star\ZZ_2 \cong \left(\pi_\star\ZZ_2/\langle 1_{\{0\}},1_V \rangle\right) \oplus \ZZ_2^2.
\]
\end{proof}

We now describe the connecting homomorphism $\beta$ via the snake lemma applied to the sequence of \u{C}ech complexes of the sheaves appearing in the short exact sequence \eqref{eq:ses_2}. We remark that the map $\beta$ coincides with the connecting homomorphism from $H^1(B,R^2f_\star\ZZ_2)$ to $H^2(B,R^1f_\star\ZZ_2)$ which appears in the long exact sequence associated to \eqref{eq:ses_2}.

We fix an open cover $\foU$ of $B$ which is a Leray cover for all the sheaves appearing in \eqref{eq:ses_2}. Assume moreover that every intersection of open sets in $\foU$ is either empty or contractible, and that no intersection of three open sets in $\foU$ intersects $\Delta$. Letting $I$ be an index set for $\foU$, and given distinct elements $i_1,\ldots i_k \in I$, we define 
\[
U_{i_1,\ldots,i_k} := U_{i_1} \cap \cdots \cap U_{i_k},
\]
where $U_i$ denotes the element of $\foU$ indexed by $i \in I$. Recall that, in general,
\[
\Gamma(U_{i_1,\ldots i_k},R^af_\star\ZZ_2) \cong H^a(f^{-1}(U_{i_1,\ldots,i_k}),\ZZ_2)
\]
for any $k \in \ZZ_{> 0}$, any subset $\{i_1,\ldots, i_k\} \subset I$ of size $k$, and any $a \in \{1,2\}$. Moreover, by $\ZZ_2$-simplicity, $H^a(f^{-1}(U_{i_1,\ldots,i_k}),\ZZ_2)$ is isomorphic to
\begin{equation}
\label{eq:identification-alphaij}
\Gamma(U_{i_1,\ldots,i_k},j_\star R^af_\star\ZZ_2) = \Gamma(U_{i_1,\ldots,i_k} \cap B_0, R^af_\star\ZZ_2)
\end{equation}
for each $a \in \{1,2\}$. Fixing a base point $p \in U_{i_1,\ldots,i_k} \cap B_0$, this group is canonically isomorphic to the subspace of $H^a(f^{-1}(p),\ZZ_2) \cong H^a(\TT^3,\ZZ_2)$, which is monodromy invariant around every loop in $U_{i_1,\ldots,i_k} \cap B_0$ based at $p$. Hence we may identify elements of $C^1(B,R^2f_\star\ZZ_2)$ with a subspace of the vector space of tuples 
\[
\alpha = \big(\alpha_{i,j} \in H^2(\TT^3,\ZZ_2)  : i,j \in I,~i\neq j\big),
\]
where each $\alpha_{i,j}$ is monodromy invariant around loops in $U_{i,j} \cap B_0$, and hence uniquely determines an element of $H^2(f^{-1}(U_{i,j}),\ZZ_2)$. Note that this identification relies on the identification of various fibres of $f$ with a fixed $\TT^3$; however, these identifications play no role in what follows. We can identify $C^2(B,R^2f_\star\ZZ_2)$ with a vector space of tuples of cohomology elements $\alpha_{i,j,k}$ in a (different) reference torus $\TT^3$ in a similar way. Note that, fixing indices $\{i,j,k\}$ and a pair $\{i,j\}$, a choice of path in $U_{i,j} \setminus \Delta$ from a point $p \in U_{i,j}$ to a point $q \in U_{i,j,k}$ determines an isomorphism between these two reference tori. While this identification depends on the choice of path if $U_{i,j} \cap \Delta \neq \varnothing$, the choice of path does not affect the identification of the monodromy invariant part of $H^a(f^{-1}(p),\ZZ_2)$ with a subspace of $H^a(f^{-1}(q),\ZZ_2)$.

\begin{lemma}
\label{lem:beta_explicit}
Let $\alpha = \big(\alpha_{i,j} \in H^2(\TT^3,\ZZ_2)  : i,j \in I,~i\neq j\big)$ be a cocycle, and let $[\alpha]$ be the class of $\alpha$ in $H^1(B,R^2f_\star\ZZ_2)$. The class 
\[
\beta([\alpha]) \in H^2(B,R^1f_\star\ZZ_2)
\]
is represented by the cocycle 
\[
\xi := (\xi_{i,j,k} : i,j,k \in I) \in C^2(B,R^1f_\star\ZZ_2)
\]
defined as follows. If $\alpha_{i,j}$, $\alpha_{i,k}$, and $\alpha_{j,k}$ are pairwise distinct, $\xi_{i,j,k} \in H^1(\TT^3,\ZZ_2) \cong V^\star$ is the unique generator of the annihilator of the span
\[
W := \langle \alpha_{i,j},\alpha_{j,k},\alpha_{i,k} \rangle \subset V = H^2(\TT^3,\ZZ_2).
\]
Otherwise, if $\alpha_{i,j}$, $\alpha_{i,k}$, and $\alpha_{j,k}$ are not all pairwise distinct, $\xi_{i,j,k} := 0$.
\end{lemma}
\begin{proof}
The long exact sequence of \u{C}ech cohomology groups associated to a short exact sequence of sheaves agrees with that defined using derived functor cohomology. To describe the image of $\beta$ in \u{C}ech cohomology we identify $\beta$ with the connecting homomorphism obtained from the snake lemma applied to the short exact sequence of \u{C}ech complexes
\begin{equation}
\label{eq:cech_complexes}
0 \lra C^\bullet(B,R^1f_\star\ZZ_2) \lra C^\bullet(B,\shF) \lra C^\bullet(B,R^2f_\star\ZZ_2) \lra 0.
\end{equation}
Recall that we chose the cover $\foU$ of $B$ such that no triple intersection $U_{i,j,k}$ of open sets intersects $\Delta$. Hence $R^2f_\star\ZZ_2$ is locally constant over $U_{i,j,k}$. Thus we have that, for all (pairwise distinct) indices $i,j,k \in I$,
\begin{eqnarray}
	\Gamma(U_{i,j,k},R^2f_\star\ZZ_2) & \cong & V \\	
	\nonumber
	\Gamma(U_{i,j,k},\pi_\star \ZZ_2) & \cong & \ZZ_2^V,
\end{eqnarray}
where $\ZZ_2^V$ denotes the set of functions from $V = H^2(\TT^3,\ZZ_2)$ to $\ZZ_2$. Note that
\[
\Gamma(U_{i,j,k}, \shF ) \cong \ZZ_2^V / \langle 1_{\{0\}} , 1_V \rangle,
\]
where $\shF$ is as described in Lemma~\ref{lem:ses_2} and $i,j,k \in I$ are pairwise distinct indices. We describe $\beta$ using the following part of the exact sequence \eqref{eq:cech_complexes}, where $\delta_i$ denote the \u{C}ech coboundary maps for $i\in\{1,2,3\}$.
\begin{equation}
\label{eq:snake_redux}
\nonumber
\xymatrix@R+1.5pc@C+2pc{
	0 \ar[r] & C^1(B,R^1f_\star\ZZ_2) \ar[r]^{\iota_1}\ar[d]_{\delta_1} & C^1(B,\F) \ar[r]^{\phi_1}\ar[d]_{\delta_2} & C^1(B,R^2f_\star\ZZ_2) \ar[r]\ar[d]_{\delta_3} & 0 \\
	0 \ar[r] & C^2(B,R^1f_\star\ZZ_2) \ar[r]^{\iota_2} & C^2(B,\F) \ar[r]^{\phi_2} & C^2(B,R^2f_\star\ZZ_2) \ar[r] & 0.
}
\nonumber
\end{equation}

Note that, given pairwise distinct indices $i,j,k \in I$, since $\alpha$ is a cocycle, we have
\[
\delta_3(\alpha)=\alpha_{i,j} + \alpha_{j,k} + \alpha_{i,k} = 0.
\]
Lift $\alpha \in C^1(B,R^2f_\star\ZZ_2)$ to the element %$\bar{\alpha} \in C^1(B,\shF)$ such that 
\[
\bar{\alpha} := \big(\bar{\alpha}_{i,j} \in \Gamma(U_{i,j},\shF  : i,j \in I,~i\neq j\big) \in C^1(B,\shF),
\]
where $\bar{\alpha}_{i,j}$ denotes the class of the indicator function $1_{\alpha_{i,j}}$ in $\Gamma(U_{i,j},\shF)$. Note that, in this notation, $\phi_1 \colon \bar{\alpha}_{i,j} \mapsto \alpha_{i,j}$.

Since
\[
\delta_2(\bar{\alpha})_{i,j,k} = \bar{\alpha}_{i,j} + \bar{\alpha}_{j,k} + \bar{\alpha}_{i,k}
\] 
for any pairwise distinct indices $i$, $j$, and $k$ in $I$, we have that $\delta_2(\bar{\alpha})_{i,j,k}$ is represented by the function 
\[
1_{\alpha_{i,j}} + 1_{\alpha_{j,k}} + 1_{\alpha_{i,k}} \in \ZZ_2^V.
\]
Following the proof of the snake lemma,
\[
\phi_2\circ\delta_2(\bar{\alpha})_{i,j,k} = (\delta_3\circ \phi_1(\bar{\alpha}))_{i,j,k} = 0,
\]
and so $\delta_2(\bar{\alpha}) = \iota_2(\zeta)$ for some cocycle $\zeta \in C^2(B,R^1f_\star\ZZ_2)$ such that
\[
\beta([\alpha])=[\zeta].
\]
It thus remains to check that we can take $\zeta$ to the be cocycle $\xi$ as described in the statement of the Lemma. We verify this in each of the following situations.
\begin{enumerate}
\item The cohomology classes $\alpha_{i,j}$, $\alpha_{j,k}$, and $\alpha_{i,k}$ are pairwise distinct. In this case the three elements $\alpha_{i,j}$, $\alpha_{j,k}$, and $\alpha_{i,k}$ lie in a unique two dimensional subspace $W$ of $\Gamma(U_{i,j,k},R^2f_\star\ZZ_2)$. Recall that -- by \eqref{eq:identification-alphaij} -- $\Gamma(U_{i,j,k},R^2f_\star\ZZ_2)$ is isomorphic to $V$. In this case $\xi_{i,j,k}$ is the unique non-zero linear map on $V$ which generates the annihilator of $W$. Regarding $\xi_{i,j,k}$ as an element of $\ZZ_2^V$, we have that
\[
(1_{\{0\}} + 1_V + \xi_{i,j,k}) = 1_{\alpha_{i,j}} + 1_{\alpha_{j,k}} + 1_{\alpha_{i,k}} = 1_{W\setminus \{0\}}.
\]
Note that, since $1_{\{0\}}$ and $1_V$ vanish in $\ZZ_2^V / \langle 1_{\{0\}}, 1_V  \rangle$,
\[
\iota_2(\xi)_{i,j,k} = \delta_2(\bar{\alpha})_{i,j,k} \in \ZZ_2^V / \langle 1_{\{0\}}, 1_V  \rangle.
\]
\item The cohomology classes $\alpha_{i,j}$, $\alpha_{j,k}$, and $\alpha_{i,k}$ are not all distinct. In this case at least one of $\alpha_{i,j}$, $\alpha_{j,k}$, and $\alpha_{i,k}$ is zero. It follows that $\delta_2(\bar{\alpha})_{i,j,k}$ is represented by the indicator function of the origin, which vanishes in the group $\ZZ_2^V / \langle 1_{\{0\}}, 1_V  \rangle$. In this case we also recall that $\xi_{i,j,k} := 0 \in V$.
\end{enumerate}
Therefore, $\iota_2(\xi)_{i,j,k}$ coincides with $\delta_2(\bar{\alpha})_{i,j,k}$ for all values of $i,j,k \in I$, as required.
\end{proof}

%The first of our main results, in the next Theorem \ref{thm:beta_is_square}, follows from the description of $\beta$ given in Lemma~\ref{lem:beta_explicit} and the identification of `mirror dual' pairs of sheaves.

We now make use of the description of $\beta\colon H^1(B, R^2f_\star\ZZ_2)  \to  H^2(B, R^1f_\star\ZZ_2)$ given in Lemma~\ref{lem:beta_explicit} to give a mirror interpretation. Our main result, Theorem~\ref{thm:beta_is_square} shows that this map is given by squaring divisor classes in the mirror Calabi--Yau, as in \cite[Theorem~$4.1$]{GrossSLagI}. To show this, we use the description of the cup product in \u{C}ech cohomology \cite[IV $6.8$]{Bredon}, which we recall in the following remark. 

\begin{remark}
	\label{rem:product}
	Given cochains $\alpha \in C^a(B,\shI)$ and $\beta \in C^b(B,\shJ)$ for sheaves $\shI$ and $\shJ$, the product $\alpha \otimes \beta \in C^{a+b}(B, \shI\otimes \shJ)$, is defined by setting
	\[
	(\alpha \otimes \beta)_{i_0,\ldots,i_{a+b}} := \sum_{c=0}^{a+b}\alpha_{i_0,\ldots,i_c}\otimes\beta_{i_c,\ldots,i_{a+b}},
	\]
	where the subscript indices are interpreted cyclically. This product is compatible with the Leray spectral sequence and the usual cup product.
\end{remark}
 
\begin{theorem}
	\label{thm:beta_is_square}
	Using the identifications given in Lemma~\ref{lem:dual_sheaves}, the connecting homomorphism $\beta$ in the long exact sequence \eqref{eq:les_CBM} is identical to the map
\begin{eqnarray}
\label{Eq: square}
\Square \colon H^1( B,R^1\breve{f}_\star\ZZ_2) & \lra & H^2( B,R^2\breve{f}_\star\ZZ_2) \\
\nonumber
D & \longmapsto & D^2
\end{eqnarray}
\end{theorem}
\begin{proof}
As described in Remark \ref{rem:product}, the cup product admits a description via \u{C}ech cohomology. Fix a cocycle
\[
\cycleD = (\cycleD_{i,j} : i,j \in I, i \neq j ) \in  C^1(B, R^1\breve{f}_\star\ZZ_2)
\]
and let $D$ denote the class it represents in $H^1( B,R^1\breve{f}_\star\ZZ_2)$. 

Analogously to our treatment of $\alpha \in C^1(B, R^2f_\star\ZZ_2)$ in \eqref{eq:identification-alphaij}, we identify each element $\cycleD_{i,j} \in \Gamma(U_{i,j}, R^1\breve{f}_\star\ZZ_2)$ with an element in $H^1((\TT^3)^\star,\ZZ_2)$. In particular, fixing a basepoint $p \in U_{i,j} \cap B_0$, we identify $\cycleD_{i,j}$ with an element of $H^1((\TT^3)^\star,\ZZ_2)$ which is monodromy invariant around loops in $U_{i,j}\cap B_0$.

Since $\cycleD$ is a cocycle, 
\[
\cycleD_{i,j}+\cycleD_{j,k}+\cycleD_{k,i}=0
\] 
for any pairwise distinct indices $i$, $j$, and $k \in I$. Hence, as in Lemma \ref{lem:beta_explicit}, we are in one of the following two situations.
\begin{enumerate}
	\item $\cycleD_{i,j} = e_1$, $\cycleD_{j,k} = e_2$, and $\cycleD_{i,k} = e_1+e_2$, for a basis $\{e_1,e_2,e_3\}$ of 
\[
H^1((\TT^3)^\star,\ZZ_2) \cong H^2(\TT^3,\ZZ_2).
\] 
By Remark~\ref{rem:product}, we have that
\begin{align*}
	(\cycleD^2)_{i,j,k} &= \cycleD_{i,j}\cycleD_{j,k} + \cycleD_{j,k}\cycleD_{i,k} + \cycleD_{i,k}\cycleD_{i,j}\\
	 &= e_1e_2 + e_2(e_1+e_2) + e_1(e_1+e_2) \\
	 &= 3e_1e_2 + e_1^2 + e_2^2.
\end{align*}
Recall that the cohomology algebra of the torus is the exterior algebra generated by $e_1,e_2,e_3$ and hence 
\begin{eqnarray}
\nonumber
e_1^2=e_2^2= 0 \textrm{, and } \\
\nonumber
3e_1e_2 = e_1e_2 = e_3^\star,
\nonumber
\end{eqnarray}
where the final equality uses the standard identification
\begin{equation}
\label{eq:identify_with_Hodge_star}
H^2((\TT^3)^\star,\ZZ_2) \cong H^1((\TT^3)^\star,\ZZ_2)^\star.
\end{equation}
In other words, $(\cycleD^2)_{i,j,k}$ is the unique non-zero covector which annihilates $\cycleD_{i,j}$, $\cycleD_{j,k}$, and $\cycleD_{i,k}$. Letting $\alpha := \mu_1(\cycleD)$, we have that (since $\mu_2$ is compatible with the identification \eqref{eq:identify_with_Hodge_star} on smooth fibres), $\mu_2$ identifies $(\cycleD^2)_{i,j,k}$ with $\xi_{i,j,k}$, where $\xi$ is the cocycle defined in the proof of Lemma~\ref{lem:beta_explicit}.
\item One of $\{\cycleD_{i,j},\cycleD_{j,k},\cycleD_{i,k}\}$ is zero, and the other two are equal. Without loss of generality, assume $\cycleD_{i,j} = 0$ and let $e := \cycleD_{j,k} = \cycleD_{i,k}$. We observe that
\begin{align*}
(\cycleD^2)_{i,j,k} &= \cycleD_{i,j}\cycleD_{j,k} + \cycleD_{j,k}\cycleD_{i,k} + \cycleD_{i,k}\cycleD_{i,j}\\
&= 0 + e^2 + 0 \\
&= 0.
\end{align*}
Note that $(\cycleD^2)_{i,j,k}$ is also evidently zero if $\cycleD_{i,k} = \cycleD_{j,k} = \cycleD_{i,k} = 0$. 
\end{enumerate}
Thus we have that $\beta(\alpha)_{i,j}$ and $\Square(\cycleD)_{i,j}$ are identified (by $\mu_2$) for any distinct $i$ and $j$ in $I$, from which the result follows.
\end{proof}

\begin{remark}
\label{Rem: number of connected components}
If $X$ and $\breve{X}$ are simply connected, $H^1(B,R^1\breve{f}_\star\ZZ_2) \cong H^2(X,\ZZ_2)$, and the map $\Square$ in Theorem \ref{thm:beta_is_square} coincides with the usual cup product $D \mapsto D \smile D$ in cohomology. Moreover, in this case the real Lagrangians $L_\RR \subset X$ and $\breve{L}_\RR \subset \breve{X}$ have two connected components \cite[Corollary $1$]{CBM}.
\end{remark}

\begin{corollary}
\label{cor:mod2_cohomology_again}
If $H^1(\breve{X},\ZZ_2) = 0$, then
\begin{equation}
\label{eq:corollary_1}
h^1(L_\RR,\ZZ_2) = h^1(B,R^1f_\star\ZZ_2) + \dim\ker(\Square),
\end{equation}
where $\Square$ is the map defined in \eqref{Eq: square}. Moreover if, in addition, $H^2(\breve{X},\ZZ) \cong \ZZ$, $H^3(\breve{X},\ZZ)$ contains no $2$-torsion, and $H^1(X,\ZZ_2) = 0$, then 
\begin{equation}
\label{eq:corollary_2}
h^1(L_\RR,\ZZ_2) = h^1(B,R^1f_\star\ZZ_2) + \delta.
\end{equation}
Here
\[
\delta = 
\begin{cases}
1 & \text{if $\overline{D}^3$ is divisible by $2$.} \\
0 & \text{otherwise,}
\end{cases}
\]
where $\overline{D}$ is a generator of $H^2(\breve{X},\ZZ)$.
\end{corollary}
\begin{proof}
Consider the following terms of the exact sequence \eqref{eq:les_CBM},
\begin{align*}
H^0(B,R^2f_\star\ZZ_2)  \to H^1(B,R^1f_\star\ZZ_2) \to H^1(L_\RR,\ZZ_2)
  \to H^1(B,R^2f_\star\ZZ_2) \stackrel{\beta}\to H^2(B,R^1f_\star\ZZ_2),
\end{align*}
where $\beta = \Square$ by Theorem~\ref{thm:beta_is_square}. The identity \eqref{eq:corollary_1} follows from the fact $H^0(B,R^2f_\star\ZZ_2)$ vanishes. Indeed, using the Leray spectral sequence for $\breve{f}$ and the vanishing of $H^2(B,\breve{f}_\star \ZZ_2)$, we obtain a surjection
\[
H^1(\breve{X},\ZZ_2) \lra H^0(B,R^1\breve{f}_\star\ZZ_2).
\]
Since, by assumption, $H^1(\breve{X},\ZZ_2) = 0$ it follows that $H^0(B,R^1\breve{f}_\star\ZZ_2) = 0$, and so \eqref{eq:corollary_1} follows from Lemma~\ref{lem:dual_sheaves}.

To prove identity \eqref{eq:corollary_2}, observe that the Bockstein homomorphism 
\[
H^2(\breve{X},\ZZ_2) \to H^3(\breve{X},\ZZ)
\]
induced by the short exact sequence
\begin{equation}
\nonumber
0 \longrightarrow \ZZ \longrightarrow \ZZ \longrightarrow \ZZ_2 \lra 0
\end{equation}
vanishes, since $H^3(\breve{X},\ZZ)$ contains no $2$-torsion. It follows that the map
\begin{equation}
H^2(\breve{X},\ZZ) \longrightarrow H^2(\breve{X},\ZZ_2)
\end{equation}
is surjective, and $H^2(\breve{X},\ZZ_2) \cong \ZZ_2$. Let $\overline{D} \in H^2(\breve{X},\ZZ)$ denote a lift of the generator $D \in H^2(\breve{X},\ZZ_2)$. Observe that the map $\Square$ defined in \eqref{Eq: square} is the zero map if and only if $D\cdot D^2 = D^3$ is zero; that is, exactly when $\overline{D}^3 \in H^6(\breve{X},\ZZ)$ is divisible by $2$. In this case
\[
\ker(\Square) = H^1(B,R^1\breve{f}_\star\ZZ_2) = H^2(\breve{X},\ZZ_2),
\]
where the latter equality follows from the Leray spectral sequence for $\breve{f}$. Specifically, this equality follows from the vanishing of the cohomology groups $H^0(B,R^2\breve{f}_\star\ZZ_2)$ and $H^2(B,\breve{f}_\star\ZZ_2)$, and the vanishing of the differential
\[
d_2 \colon H^1(B,R^1\breve{f}_\star\ZZ_2) \to H^3(B,\ZZ_2)
\]
on the $E_2$ page of the Leray spectral sequence for $\breve{f}$. The cohomology group $H^0(B,R^2\breve{f}_\star\ZZ_2)$ vanishes since
\[
H^0(B,R^2\breve{f}_\star\ZZ_2) \cong H^0(B,R^1f_\star\ZZ_2) = 0,
\]
while $H^2(B,\breve{f}_\star\ZZ_2)$ vanishes since $B$ is a $\ZZ_2$-homology sphere. The differential $d_2$ vanishes by the argument used in the proof of \cite[Lemma~$2.4$]{GrossSLagI}.
\end{proof}

\begin{example}
Let $\breve{f}\colon \breve{X} \to B$ be the mirror to the quintic threefold $X$, together with the fibration described in \S\ref{sec:real_locus}. Note that $\breve{X}$ satisfies the conditions of Corollary~\ref{cor:mod2_cohomology_again}, where $\delta = 0$, since $H^2(\breve{X},\ZZ)$ is generated by a hyperplane class $H$ in $\PP^4$ and $H^3$ is represented by $5$ points in $\breve{X}$. Thus we obtain
 \begin{equation}
 \label{mirror quintic}
 h^1(\breve{L}_{\RR},\ZZ_2) = 101,
  \end{equation}
where $\breve{L}_{\RR} \subset \breve{X}$ denotes the real Lagrangian in $\breve{X}$ described in \S\ref{sec:real_locus}.
\end{example}

\begin{example}
\label{eg:quintic}
The real locus $L_{\RR} \subset X$, in the quintic threefold, with the fibration described in \ref{QuinticBase}, satisfies
\begin{eqnarray}
h^1(L_\RR,\ZZ_2) = h^2(L_\RR,\ZZ_2) = 29.
\nonumber
\end{eqnarray}
This follows by an explicit calculation of the map $\Square$ in Theorem~\ref{thm:beta_is_square}, made via consideration of the triple intersection form on $H^2(\breve{X},\ZZ)$, as described in \cite[Proposition~$4.2$]{GrossTopology}. To compute this we construct a $101\times 101$ matrix whose rank is equal to the rank of the map $\Square$ using the computer algebra system MAGMA \cite{MAGMA}. Documented source code for this calculation, which can be adapted to further examples where the intersection form is known, is available at \cite{Prince}.
\end{example}

%% file: lagrangian_topology.tex
% !TEX root = paper.tex
%----------------------------------------------------------------------
\section{Triangulations, flips and Dehn surgery}
\label{sec:flips}

\subsection{Maximal triangulations}
\label{Maximal triangulations}

We recall the notion of \emph{maximal triangulation} and its connection to the construction of $\ZZ$-simple torus fibrations. We concentrate particularly on the example of the quintic, studied in detail by Gross in \cite{GrossTopology}. In this construction, the discriminant locus is contained inside the two-dimensional faces of a (four-dimensional) polytope, and its restriction to each two-dimensional face $F$ is the dual complex to the one-skeleton of a fixed triangulation of $F$.

A maximal triangulation $\P$ of a polytope $\Xi\subset M_\RR$ is a simplicial decomposition of $\Xi$, such that the set of $0$-simplices is equal to the set integral points $\Xi \cap M$. A triangulation is \emph{projective} if there exists a strictly convex height function $h_\Xi \colon M_\RR \to \ZZ$ which is linear on each simplex $\tau \in T\Xi$, where $T\Xi$ denotes the affine tangent space of $\Xi$. By \cite[Proposition~$3$]{GKZ}, there exists a maximal projective triangulation on any lattice polytope $\Xi \subset M_\RR$. Moreover, a maximal projective triangulation of $\Xi$ determines an induced triangulation on the boundary $\partial \Xi$. Applying \cite[Theorem~$3.16$]{GrossBB} (see also \cite{HZci}), we can obtain \emph{simple} integral affine structure from any maximal projective triangulation. In the following example, we describe the base $B$ of a fibration of the quintic threefold as an integral affine manifold with simple singularities -- see also \cite[\S$19.3$]{GHJ} and \cite[Example~$6$]{CBM}. 

\begin{example}
\label{QuinticBase}
 Let $\Delta_{\PP^4}$ be the moment polytope of the toric variety $\PP^4$, with the usual anti-canonical polarization. $\Delta_{\PP^4}$ is the $4$-simplex given by the convex hull of the points 
\[P_1  =  (-1,-1,-1,-1) \,\ \,\ 
P_2  =  (4,-1,-1,-1) \,\ \,\ 
P_3  =  (-1,4,-1,-1) \]
\[
P_4  =  (-1,-1,4,-1) \,\ \,\ 
P_5  =  (-1,-1,-1,4). \]
We let $B := \partial \Delta_{\PP^4}$, denote the boundary of $\Delta_{\PP^4}$. That is, $B$ is the union of $5$ tetrahedra, glued along their common triangular faces. We illustrate $B$, together with a number of positive and negative vertices of the discriminant locus (in red) in Figure~\ref{fig:Base}.
\begin{figure}
	\resizebox{0.8\textwidth}{!}{
		\includegraphics{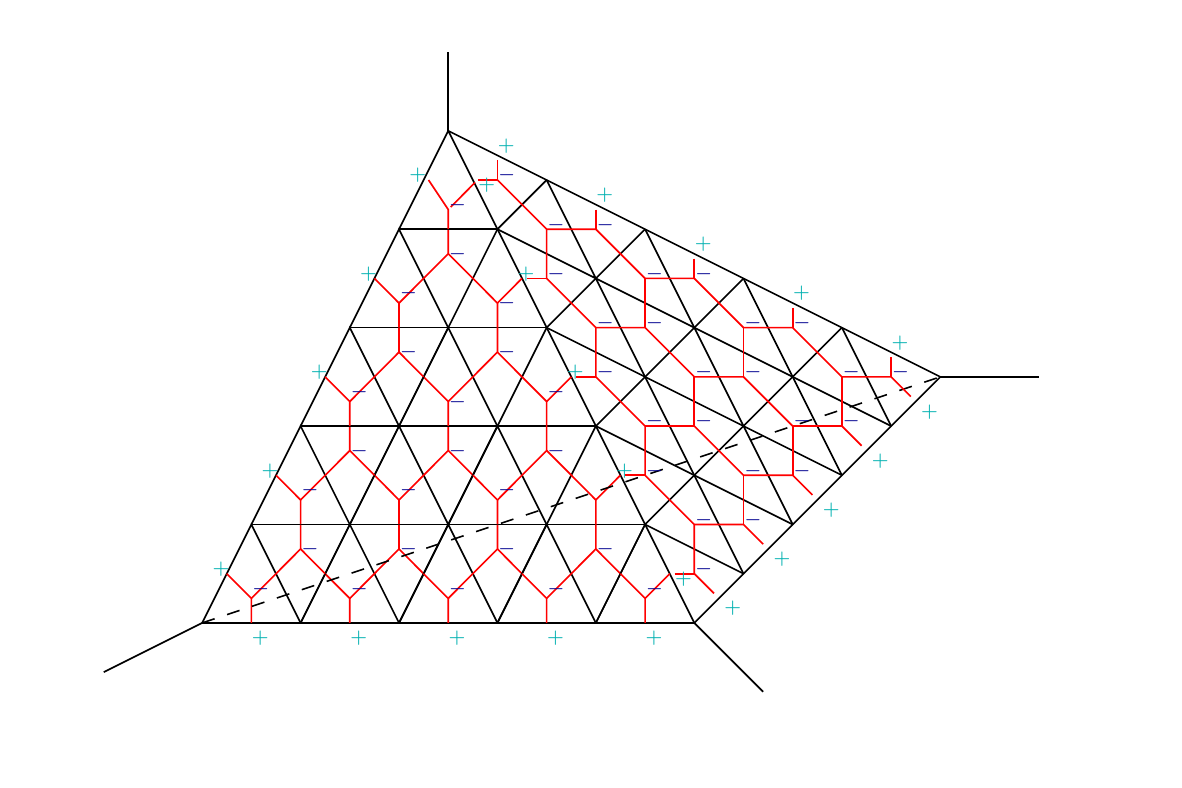}
	}
	\caption{Positive and negative vertices of $\Delta$.}
	\label{fig:Base}
\end{figure}
Note that the triangulation of $B$ induced by its description as the boundary of a polytope contains ten triangular faces, ten edges, and five vertices. We also note that $B$ is homeomorphic to a $3$-sphere.
Let $\sigma_{ijk}$ denote the triangular face of $\Delta_{\PP^4}$ spanned by the vertices $P_i$, $P_j$, and $P_k$. Fix a maximal projective triangulation of $\Delta_{\PP^4}$, and hence an (integral) triangulation of each face $\sigma_{ijk}$. Let $\Delta_{ijk}$ be the union of the one dimensional cells in the first barycentric subdivision of this triangulation of $\sigma_{ijk}$ which do not contain an integral point of $\sigma_{ijk}$. We note that the integral affine structure is determined by the triangulation on each triangular face $\sigma_{ijk}$. Finally, we fix the discriminant locus
\begin{equation}
\label{TheDiscriminantLocus}
\Delta:= \bigcup_{i,j,k} \Delta_{ijk}.
\nonumber
\end{equation}
The affine structure on $B_0 := B\setminus \Delta$ is described, for example, in \cite[p.~157]{GHJ}. Over the interior of each tetrahedron in $\partial \Delta_{\PP^4}$, the integral affine structure is obtained by restriction from the embedding $\Delta_{\PP^4} \hookrightarrow \RR^4$. 

We describe the affine monodromy $T_\gamma \in Gl_n(\ZZ)$, see Definition \ref{Def: affine monodromy}, around certain loops $\gamma\colon S^1 \to B \cong \partial \Delta_{\PP^4}$ based at a vertex $v$ of $\Delta_{\PP^4}$. First observe that, in an affine neighbourhood of $v$, the strata of $\partial \Delta_{\PP^4}$ containing $v$ form a fan isomorphic to the fan of $\PP^3$. We identify this neighbourhood of $v$ with a domain in $\RR^3$ with its standard integral affine structure, and let $\{e_i : i \in \{1,2,3\}\}$ denote the standard basis of $\RR^3$. We identify the ray generators of the edges of $\partial \Delta_{\PP^4}$, emanating from $v$, with the vectors
\[
\{ e_1,e_2,e_3,-e_1-e_2-e_3\}.
\]
We let $d_i := e_i$ for $i \in \{1,2,3\}$, and $d_4 := -e_1-e_2-e_3$. We also let $\tau_i$, $i \in \{1,\ldots,4\}$ denote the edge of $\Delta_{\PP^4}$ containing $v$ with tangent direction $d_i$ at $v$, and let $\sigma_{ij}$ denote the face of $\Delta_{\PP^4}$ containing the edges $\tau_i$ and $\tau_j$, for $i$,~$j \in \{1,\ldots,4\}$. Consider segments of the singular locus $\Delta$ contained in $\sigma_{ij}$ for $i$,~$j \in \{1,\ldots,4\}$, as shown in Figure~\ref{fig:loops}. We choose loops $\gamma_{ij,k}$, where $k \in \{i,j\}$, based at $v$ passing singly around the segment of the singular locus contained in $\sigma_{ij}$ which meets the ray indexed by $k$. Examples of such loops $\gamma_{ij,k}$ are also illustrated in Figure~\ref{fig:loops}. We orient the loops $\gamma_{ij,k}$ by insisting that the tangent vector of $\gamma_{ij,k}$ at the unique point (other than $v$) at which the image of $\gamma_{ij,k}$ intersects $\sigma_{ij}$ pairs positively with the normal vector $n_{ij}$ to $\sigma_{ij}$. Note that $n_{ij}$ is uniquely determined since $\sigma_{ij}$ is oriented, as the ray generators lying on $\tau_i$ and $\tau_j$ respectively form an ordered basis of the linear span of $\sigma_{ij}$. The monodromy in $TB_0$ around $\gamma_{ij,k}$ is given by the map
\[ 
v  \longmapsto v + \langle n_{ij},v \rangle \cdot d_k,
\]
as in \cite[\S$6.4$]{DBranes09}, and we denote the corresponding matrix by $T_{ij,k}$. For example, $T_{12,1}$ is given by the linear map $v \mapsto v + \langle (0,0,1),v \rangle \cdot (1 , 0,  0)$ and is represented by the matrix
\begin{align*}
\label{Eqn: monodromy matrices}
T_{12,1}  =  \left( \begin{matrix} 1&0&1\\ 0&1&0 \\0&0&1\\ \end{matrix}\right).
\end{align*}
Note that the face $\sigma_{ij}$ is invariant under the linear transformation $T_{ij,k}$. In Figure~\ref{fig:loops} we illustrate examples of various $\gamma_{ij,k}$ passing around an edge of the discriminant locus $\Delta$ adjacent to both a negative and a positive vertex. The monodromy matrices around edges adjacent to two negative vertices are described analogously.

We compute the action of each $(T^t_{ij,k})^{-1}$ on the $2$-torsion points $(\frac{1}{2}\ZZ^3)/\ZZ^3$
of the real two torus $\Hom(\Lambda_b,U(1))\cong \RR^3/\ZZ^3$, where $\Lambda$ is the sheaf of integral tangent vectors. This set is in one-to-one correspondence with the set of the following eight vectors.
\begin{align*}
u_0 & = (0, 0, 0),\quad 
u_1  =(0, 0, 1),\quad
u_2  =(1, 0, 1 ),\quad
u_3  =(1, 0, 0),\\
u_4 & =(1, 1, 0),\quad
u_5  =(0, 1, 0),\quad
u_6  =(0, 1, 1),\quad
u_7  =(1, 1, 1).
\end{align*}
Note that each $u_i \in \pi^{-1}(p)$ corresponds in Figure~\ref{fig:monodromy_cube} to a vertex labelled with $i$. 
\begin{figure}
	\includegraphics{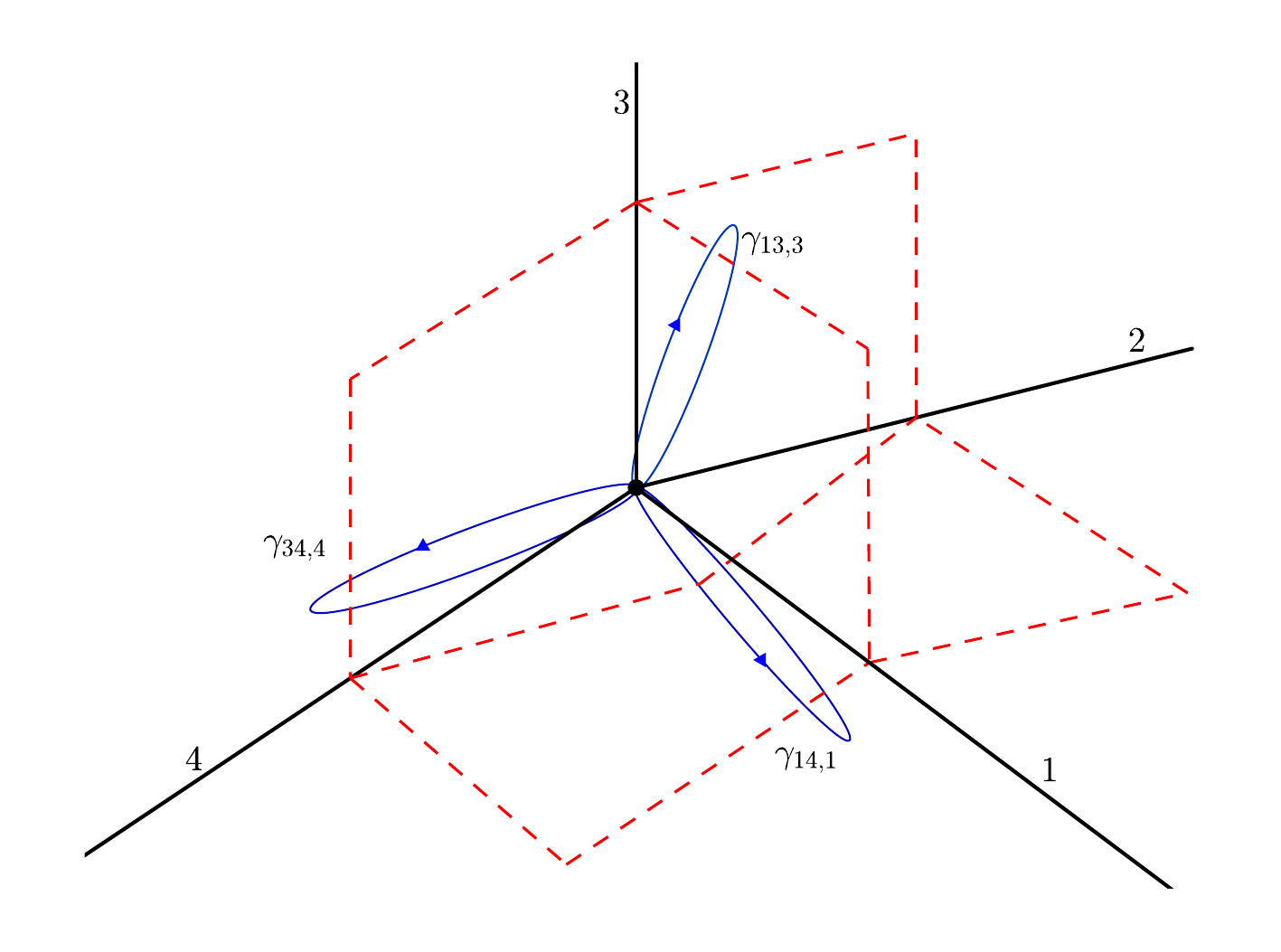}
	\caption{Examples of loops $\gamma_{ij,k}$.}
	\label{fig:loops}
\end{figure}
Observe that $u_0$ remains invariant under all the monodromy transformations $T'_{ij,k} := (T^t_{ij,k})^{-1}$. That is, $u_0$ defines a section of $f \colon X\to B$. Thus there is a connected component of the real Lagrangian $L_\RR \subset X$ homeomorphic to $S^3$. Considering the action of $T'_{ij,k}$ on the torsion points $u_i$, for $i=\{1,\ldots,7\}$, the permutations induced on the indices of these seven vectors form the following set of double transpositions: 
\begin{align*}
T'_{12,1}:(23)(47),  \quad & T'_{12,2}:(47)(56) \quad T'_{13,1}:(27)(34), \quad T'_{13,3}:(16)(27) \\ 
T'_{14,1}:(24)(37),  \quad & T'_{14,4}:(15)(37) \quad T'_{23,2}:(45)(67), \quad T'_{23,3}:(12)(67) \\ 
T'_{24,2}:(46)(57),  \quad & T'_{24,4}:(13)(57) \quad T'_{34,3}:(17)(26), \quad T'_{34,4}:(17)(35).
\end{align*}

Note that the triangulation on $B = \partial \Delta_{\PP^4}$ is maximal projective. There are, of course, many choices of such triangulations on $B=\partial \Delta_{\PP^4}$. Moreover, any choice of maximal projective triangulation corresponds to a choice of a crepant resolution of a hypersurface in the toric variety associated to the fan over the faces of ${\Delta}_{\PP^4}$, the mirror proposed in \cite{BBci,GP90}. All such crepant resolutions of the mirror hypersurface are related to each other by applying flips to the triangulation as described in Appendix~\ref{sec:Dehn_surgery}. In the remaining part of this section we show the invariance of the ranks of the mod $2$ cohomology groups of the real Lagrangians in mirror hypersurfaces under flips.
\end{example}

\subsection{Flipping the triangulation}
\label{sec:Dehn_surgery}
Let $B$ be an integral affine manifold with simple singularities, endowed with a maximal triangulation $\P$. We let $f$ and $\breve{f}$ denote the compactified torus fibrations, as described in \S\ref{sec:real_locus}. We describe the effect of a \emph{flip} of a triangulation on the real loci $L_\RR$ and $\breve{L}_\RR$. Since these descriptions are similar, and Theorem~\ref{Invariance under flips} applies to $\breve{L}_\RR$, we focus on this case.

\begin{figure}
	\includegraphics[scale=1.2]{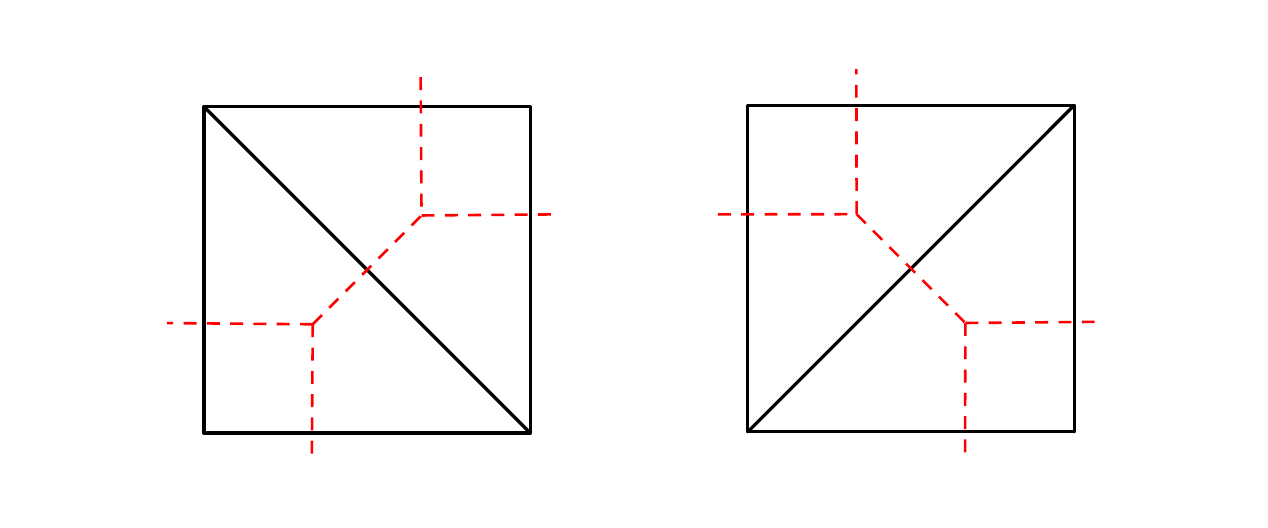}
	\caption{A flip of the triangulation $\P$}
	\label{fig:flip}
\end{figure}

\begin{definition} 
\label{def:flips}
Let $E$ be an edge of the triangulation $\P$ of $B$ such that $\Delta$ intersects $E$ in a single point, contained in a segment of $\Delta$ between a pair of negative vertices. Assume moreover that these two negative vertices are contained in two-dimensional faces $\tau_1$ and $\tau_2$ of $\P$ respectively. Moreover, we assume that the union $\tau = \tau_1 \cup \tau_2$ forms a quadrilateral which is strictly convex at its four vertices. A \emph{flip} of $\P$ is obtained by replacing $E$ with the diagonal between the two vertices of $\tau$ not contained in $E$. This induces a transformation of $\Delta$, illustrated in Figure~\ref{fig:flip}.
\end{definition}

We recall that, in considering the mirror fibration $\breve{f}$ to $f$, we do not change the integral affine structure on $B$, but the map $\breve{f}$ compactifies the torus fibration $TB_0/\Lambda \to B_0$, rather than $T^\star B_0/\breve{\Lambda}$, recalling that $\Lambda$ denotes the lattice of integral tangent vectors.

\begin{lemma}
\label{lem:conn_components}
Let $U \subset B$ be a sufficiently small neighbourhood of a segment of $\Delta$ between two negative vertices such that the closure $\shB$ of $U$ is homeomorphic to a $3$-ball. The preimage $\breve{\pi}^{-1}(\mathcal{B})$ has a connected component homeomorphic to a solid torus, and every other connected component of it is homeomorphic to a $3$-ball.
\end{lemma}

\begin{proof}
Consider two negative vertices of $\Delta$, connected by a line segment $e \subset \Delta$. Choose $U$ sufficiently small, such that it contains no trivalent points of $\Delta$ other than these two vertices and such that $\partial \mathcal{B} \cap \Delta$ consists of $4$ points, $y_1,y_2,y_3,y_4$.

Fix a base point $z \in B_0 \cap \partial \shB$, recalling that $B_0 = B \setminus \Delta$ denotes the smooth locus of $B$. We let $\{p_0, \ldots, p_7\}$ denote the points in $\breve{\pi}^{-1}(z)$, identifying $p_i$ with the point with label $i$ in the middle image of Figure~\ref{fig:monodromy_cube}. Let $\gamma_{y_i}$ denote a loop in $\partial \shB$, based at $z$ and tracing singly around the point $y_i$. By the monodromy computation in \S\ref{sec:real_locus}, the monodromy actions on $\breve{\pi}^{-1}(z) = \{p_0,\ldots, p_7\}$ corresponding to these loops are given by the following double transpositions.
\begin{align*}
	\gamma_{y_1}: (56)(47) && 	\gamma_{y_2}: (45)(67) && \gamma_{y_3}: (56)(47) && \gamma_{y_4} : (45)(67).
\end{align*}
Analysing the monodromy action around the loops $\gamma_{y_1}$ and $\gamma_{y_2}$, we see that the points $\{p_4, p_5, p_6, p_7\}$ lie in the same connected component $L_\shT$ of $\breve{\pi}^{-1}(\mathcal{B})$. Moreover, the four points $\{p_0,p_1,p_2,p_3\}$ are invariant under the monodromy action, and hence lie in sections of $\pi$ homeomorphic to $\shB$. Hence $L_\shT$ is a $4$-to-$1$ branched cover of $\shB$, branched over the part of the discriminant locus illustrated by red in Figure \ref{fig:meridians_in_B}. Restricting $\breve{\pi}$ to the boundary of $\shB$, the branch points $\{y_1,y_2,y_3,y_4\}$ are labelled in Figure~\ref{fig:complex_base}, which shows a triangulation of $\partial\shB$.

We note there are $8$ ramification points of index $2$ over $\partial \mathcal{B} \cong S^2$, and it follows from a standard Riemann--Hurwitz calculation that $\partial L_{\shT}$ is homeomorphic to the $2$-torus $\shT^2$. We identify $\shB$ with the standard $3$-ball, and fix a function $h \colon \shB \to [0,1]$ such that $h^{-1}(t) \cong S^2 \subset \shB$ for all $t \in [0,1)$, and $h^{-1}(1)=e$, where $e$ is the edge of the discriminant locus illustrated in Figure~\ref{fig:meridians_in_B}. Moreover, we require $h^{-1}(t)$ to intersect each line segment of $(\Delta \cap \shB) \setminus \{e\}$ transversely at a single point, for all $t \in [0,1)$. Thus the composition $h \circ \breve{\pi}|_{L_\shT}$ is a trivial $\mathbb{T}^2$ fibration over $[0,1)$. In other words, there is a homeomorphism $L_\shT \setminus C \to \mathbb{T}^2 \times [0,1)$, where $C$ is the circle $\breve{\pi}^{-1}|_{L_\shT}(e)$; it follows that $L_\shT$ is homeomorphic to a solid torus.
\end{proof}

\begin{proposition}
	\label{pro:flip}
	Performing a flip on the triangulation on $B$ induces a Dehn surgery on $\breve{L}_\RR$ with coefficient $2$.
\end{proposition}
\begin{proof}

\begin{figure}
	\includegraphics{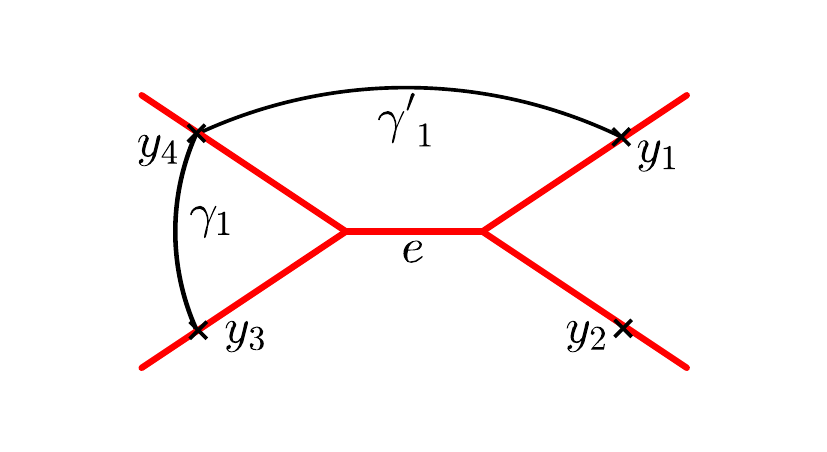}
	\caption{Curves $\gamma_1$ and $\gamma'_1$ in $\partial\shB$.}
	\label{fig:meridians_in_B}
\end{figure}

Recall that $L_\mathcal{T} \subset \breve{L}_\RR$ denotes the connected component in $\breve{\pi}^{-1}(\mathcal{B})$ homeomorphic to a solid torus. We let $\tilde{\pi}$ denote the restriction of $\breve{\pi} \colon \breve{L}_\RR \to \shB$ to $L_\shT$. Let $\gamma_1$ be the path contained $\partial \shB$ illustrated in Figure~\ref{fig:complex_base} and Figure~\ref{fig:meridians_in_B}, such that $\gamma_1(0) = y_3$, $\gamma_1(1) = y_4$. We will show that $\tilde{\pi}^{-1}(\gamma_1)$ is a meridian curve in $\partial L_\mathcal{T}$ of class $m$, while the inverse image of $\tilde{\pi}^{-1}(\gamma'_1)$ is of class $2l-m$, where $l$ denotes the class of a longitude.

To show this, we choose triangulations of $\partial L_\shT \cong \mathbb{T}^2$, and $\partial \mathcal{B} \cong S^2$, as illustrated in Figure~\ref{fig:complex} and Figure~\ref{fig:complex_base} respectively. Note that the points labelled $Y_i$ (respectively $A_i$) in Figure~\ref{fig:complex} map to points $y_i$ (respectively $a_i$) in Figure~\ref{fig:complex_base}. Hence $\tilde{\pi}^{-1}(\gamma_1)$ is the curve in $\partial L_\mathcal{T}$ given by the union of vertical line segments between the points labelled $Y_3$ and $Y_4$.

Observe that the curve $\tilde{\pi}^{-1}(\gamma_1)$ bounds a disc in $L_\shT$, given by the preimage of a triangular region bounded by $\gamma_1$ and the segments of $\Delta \cap \shB$ which contain $y_3$ and $y_4$. It follows that the curve $\tilde{\pi}^{-1}(\gamma_1)$ is a meridian curve for $L_\shT$.

Let $\gamma'_1  \subset \partial \mathcal{B}$ be the curve with $\gamma'_1(0)=y_4$, and $\gamma'_1(1)=y_1$ which is shown in Figure~\ref{fig:meridians_in_B} and Figure~\ref{fig:complex_base}. It follows from the definition of a flip, as illustrated in Figure~\ref{fig:flip}, that curve $\tilde{\pi}^{-1}(\gamma'_1)$ is a meridian curve after performing a flip. Hence, it suffices to show that $\tilde{\pi}^{-1}(\gamma'_1)$ represents a curve of class $2l-m$. To show this, consider the curve $\gamma_2 \in \partial \mathcal{B}$ with $\gamma_2(0)=y_3$, and $\gamma_2(1)=y_1$, as illustrated in Figure~\ref{fig:complex_base}. The inverse image of $\gamma_2$ under $\tilde{\pi}$ is illustrated in Figure~\ref{fig:complex} by the union of the horizontal line segments between $Y_1$ and $Y_3$. Each of these curves defines a choice of longitude in $\partial L_\shT$, and an orientation of $\tilde{\pi}^{-1}(\gamma_2)$ represents a curve of class $2l$. Note that $\gamma_2$ is homotopic to the concatenation of $\gamma_1$ and $\gamma'_1$ in $\partial \shB$. Hence, $\tilde{\pi}^{-1}(\gamma'_1)$ is of class $2l-m$, from which the result follows.

\begin{figure}
	\includegraphics[scale=1.0]{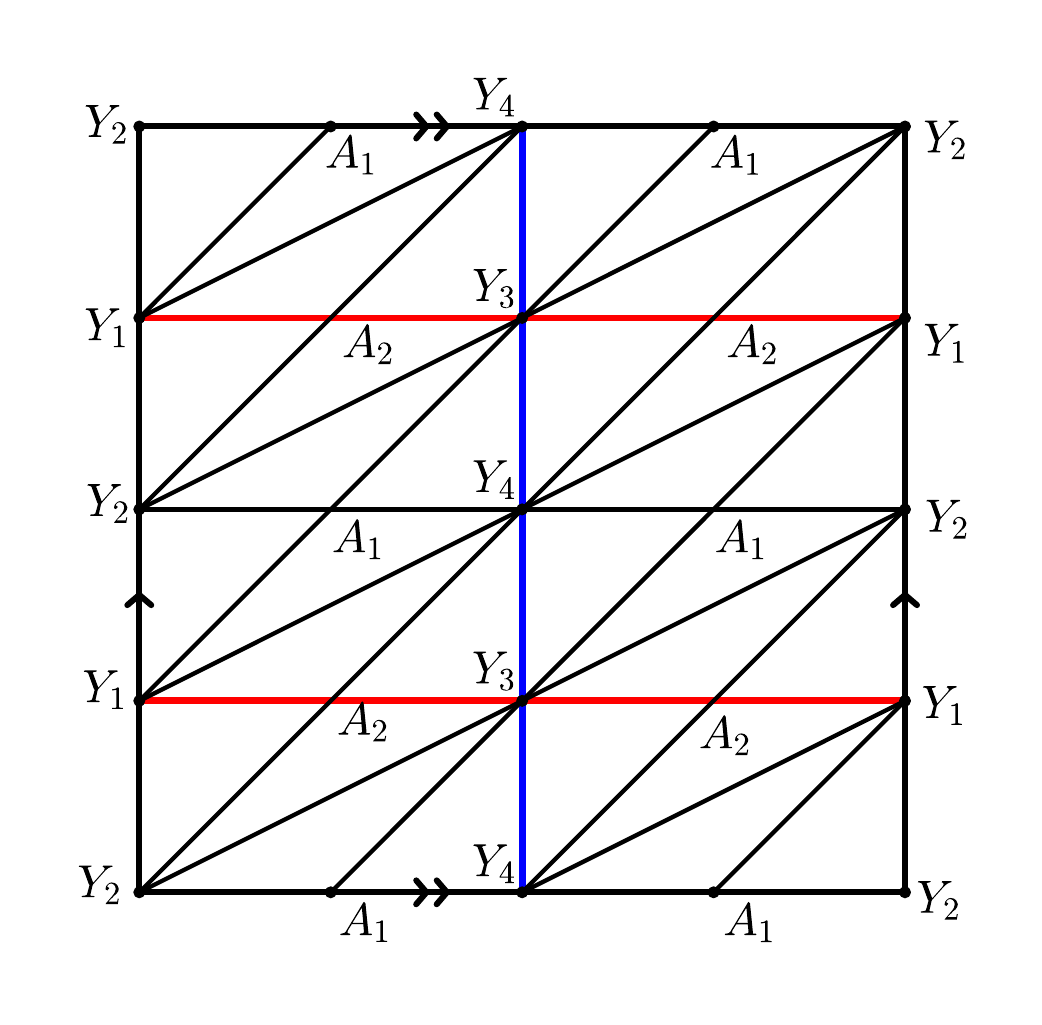}
	\caption{Triangulation of $\partial L_\mathcal{T}$.}	
	\label{fig:complex}
\end{figure}

\begin{figure}
	\includegraphics{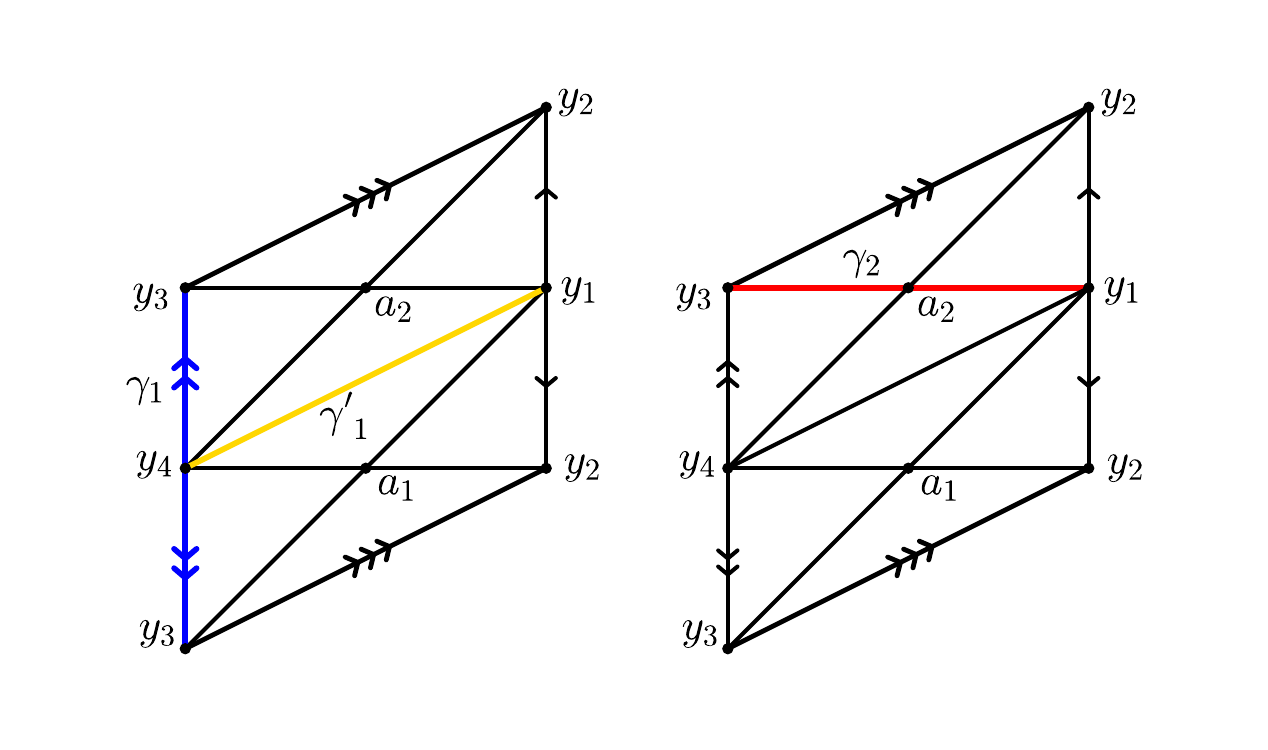}
	\caption{Triangulation of $\partial \mathcal{B}$ showing curves $\gamma_1$, $\gamma'_1$, and $\gamma_2$.}	
	\label{fig:complex_base}
\end{figure}

%\begin{figure}
%	\includegraphics{longitude_revised.pdf}
%	\caption{Triangulation of $\partial \mathcal{B}$ showing curve $\gamma_2$.}	
%	\label{fig:longitude}
%\end{figure}

\end{proof}

\begin{theorem}
\label{Invariance under flips}
	The dimension of $H^1(\breve{L}_\RR,\ZZ_2)$ remains invariant under flipping the triangulation on $B$.
\end{theorem}
\begin{proof}
	Consider the following Mayer--Vietoris sequence, associated to $L_\shT$ and its complement in $\breve{L}_\RR$,
	\[
	H_1(\partial L_\shT,\ZZ_2) \lra H_1(\breve{L}_\RR \setminus L_\shT, \ZZ_2)\oplus \ZZ_2 \lra H_1(\breve{L}_\RR,\ZZ_2).
	\]
	Note that $\dim H_1(\breve{L}_\RR \setminus L_\shT, \ZZ_2)$ is invariant under Dehn surgery, and hence surgery only changes $\dim H_1(\breve{L}_\RR,\ZZ_2)$ if the rank of the first map is different before and after the Dehn surgery. Observe that the kernel of this map is contained in the subspace spanned by the class of a meridian curve. However, since $2l+m = m \in H_1(\partial L_\shT,\ZZ_2)$, the rank of this map is unaffected by the Dehn surgery.
\end{proof}

\begin{remark}
	Repeating the analysis for $L_\RR$, we find that there is again a unique component in $\pi^{-1}(\mathcal{B})$ which is homeomorphic to a solid torus. In this case performing a flip induces a Dehn surgery of coefficient one.
\end{remark}

\begin{remark}
	In addition to flipping there is another important operation visible on the level of the torus fibration, \emph{conifold transition}. This operation changes the topological type of the total space of the torus fibration, and its connection with affine structures was explored by Casta\~no-Bernard--Matessi in~\cite{CBM3}. It follows directly from our analysis in this section that this operation also induces a Dehn surgery on the corresponding real Lagrangians.
\end{remark}

%% file: kato_nakayama.tex
%!TEX root = paper.tex
%----------------------------------------------------------------------
\section{Real Lagrangians in toric degenerations}
\label{sec:Kato-Nakayama}

This Appendix has considerable overlap with the paper~\cite{AS}, where we outline an alternative construction of the real Lagrangians, using the \emph{Kato--Nakayama space} \cite{KN}. For this we consider real toric degenerations of Calabi--Yau varieties, and describe the topology of the real loci in these degenerations, explicitly via the Kato-Nakayama space of the central fibre, regarded as a log space. This offers a number of advantages; in particular it replaces a single real Lagrangian in a Calabi--Yau compactification with a moduli space of real loci. Moreover, this construction is intrinsic to the Gross--Siebert algorithm~\cite{GS2}, which is an algebro-geometric approach the SYZ conjecture in mirror symmetry. It thus provides a systematic way of introducing real structures into a fundamental mirror symmetry construction. 

One of the main achievements of this construction is to build the coordinate ring of a mirror family of Calabi--Yau varieties, starting from combinatorial data on an integral affine manifold with singularities, $B$ \cite{GS2}. For a broad overview of this data see \cite[\S$1.1$]{ruddat2019period}. The idea is to first construct from $B$ a toric log Calabi--Yau space $X_0$, that will serve as the special fiber of the family, whose intersection complex is given by $B$. A toric log Calabi--Yau space is, in this context, a union of toric varieties glued pairwise torically along toric prime divisors, with the requirement that around the zero dimensional strata $X_0$ is toroidal \cite{GS3}. We endow $X_0$ with a log structure, obtained from the \emph{gluing data} on $B$. As it forms distinct machinery, we will not recall the theory of log structures here, and encourage the interested reader to look at \cite{K,O}. However, we discuss the gluing data shortly. 

The main results of \cite{GS2} provide a systematic way, based on a wall--crossing algorithm taking place on $B$, to build the homogeneous coordinate ring for the family of Calabi--Yau varieties with degenerate limit $X_0$, when $B$ is \emph{positive} and \emph{simple}. We refer to \cite{GHS} for a version of this result in greater generality. We recall that simple integral affine manifolds were described in \S\ref{sec:integral_affine_manifolds}, while positive integral affine manifolds are defined in \cite[\S$1.5$]{GS1}, and this condition ensures that the affine structure comes from a degeneration situation. The family obtained by this construction in \cite{GS2}, with special fiber equal to the toric log Calabi--Yau space $X_0$, is referred to as a \emph{toric degeneration} \cite{GS3}. Note that the construction of the toric degeneration from $B$ can be carried out either by viewing $B$ as the intersection complex of a toric log Calabi--Yau space, or as the dual intersection complex, as explained in \cite[\S$2$]{GS1}. $B$ is viewed as the former in \cite{AS}, which roughly equates with considering the cotangent lattice on $T^\star B_0$, rather than the tangent lattice, as in \cite[\S$5$]{GS1}. The construction of a toric degeneration from $B$ depends on choices of \emph{lifted gluing data}. This data is given by elements
\[ 
s \in H^1(B, \iota^\star \breve{\Lambda} \otimes \CC^\star),
\]
where $B_0:=B \setminus \Delta$, $\iota\colon B_0 \hookrightarrow B$ is the inclusion map, and $\breve{\Lambda} \subset T^\star B_0$ is the lattice integral cotangent vectors, see \cite[\S$5$]{GS1}. It is shown in \cite[Theorem~$5.4$]{GS1} that one can obtain a moduli space of toric log Calabi--Yau spaces $X_0$ with intersection complex $B$ from a simple and positive integral affine manifold with singularities $B$, and this moduli space is parametrised by the lifted gluing data on $B$. Lifted gluing data induces both open gluing data, which determines $X_0$ as a log space by \cite[Theorem~$5.2$, Theorem~$5.4$]{GS1}, and closed gluing data, which is the data parametrising the gluings of irreducible components of $X_0$. Hence lifted gluing data determines not only the log structure, but also the scheme structure on $X_0$ \cite[Remark~$2.33$]{GS1}. For precise definitions of lifted, open and closed gluing data we refer to \cite[\S$5$]{GHS}.

Under mild assumptions on the lifted gluing data, the toric log Calabi--Yau space $X_0$ can be constructed as a projective scheme \cite[Theorem~$2.34$]{GS1}, and in this case the reconstruction result, as outlined in \cite[\S A.$2$]{GHS}, gives rise to a projective family. In particular, in the projective setting, the general fiber is not only a complex scheme, but also admits a natural symplectic structure, and the real locus is a real Lagrangian submanifold. Note that Calabi--Yau threefolds we consider throughout this paper automatically satisfy the necessary projectivity assumption in \cite[Theorem~$2.34$]{GS1}.

In \cite{AS}, the first named author and Siebert study the real locus in the general fiber of a toric degeneration, by studying the \emph{Kato--Nakayama space} or the \emph{Betti realization} over the special fiber $X_0$. The Kato--Nakayama space of a log scheme $X$, denoted by $X^{\mathrm{KN}}$, is a topological space constructed in \cite{KN}, and the topology on $X^{\mathrm{KN}}$ can be defined as in \cite[Chapter V, Definition~$1.2.4$]{O}. We summarise some of the results of \cite{AS} in the following theorem.

\begin{theorem}
\label{Thm: AS}
Let $\shX \to \Spec R$ be a toric degeneration with special fiber $X_0$, where $R$ is a discrete valuation ring. 
\begin{itemize}
\item[i)]
\label{i} 
There is a map $\delta\colon X_0^{\mathrm{KN}} \to  S^1$, such that a general fiber of an analytic model of $\shX$ is homeomorphic to the restriction of the Kato--Nakayama space $X_0^{\mathrm{KN}}$ to $\delta^{-1}(\phi)$, denoted by $X_0^{\mathrm{KN}}(\phi)$, for any $\phi \in S^1$  \emph{\cite[\S$4$]{AS}}.
\item[ii)]
\label{ii} 
There is a natural projection map $\pi: X_0^{\mathrm{KN}} \to X_0$, and  an abstract moment map $\mu \colon X_0 \to B$, such that the composition
\[
\mu \circ \pi \colon X_0^{\mathrm{KN}} \to B
\]
is a topological torus fibration. This induces a torus fibration on $X_0^{\mathrm{KN}}(\phi)$ \emph{\cite[\S$4$]{AS}}.
\item[iii)] 
\label{iii} 
The toric log
Calabi-Yau space $(X_0,\M_{X_0})$ defined by lifted gluing data is a real log space, and carries a standard real structure, if
and only if the lifted gluing data lies in the image of
\[
H^1(B, \iota_\star\breve\Lambda\otimes\RR^\times) \lra 
H^1(B,\iota_\star \breve\Lambda\otimes\CC^\times).
\]
\emph{\cite[Corollary $4.21$]{AS}}
\item[iv)]
\label{iv}
 A real structure on the toric log Calabi--Yau space $X_0$ lifts to a real structure on $X_0^{\mathrm{KN}}$, and induces a real structure on $X_0^{\mathrm{KN}}(1)$ \emph{\cite[\S$3$]{AS}}.
\item[v)]
\label{v}
The restriction of $X_0^{\mathrm{KN}}(1)$ to the real locus, exhibits the real locus in $\shX_t$ as a surjection over $B$, with finite fibers. Away from the discriminant locus this surjection is a topological covering map with fibres of cardinality $2^n$, where $n=\dim B$ \emph{\cite[Proposition~$4.23$]{AS}}.  
\end{itemize}
\end{theorem}
It is natural to ask if given an integral affine manifold with singularities $B$, for every choice of lifted gluing data determining $X_0$, the Kato--Nakayama space $X_0^{\mathrm{KN}}(1)$ can be obtained as a topological Calabi--Yau compactification analogous to \cite{GrossTopology}. Conjecturally, such a compactification exists -- see \cite[Introduction]{GrossBB}. In dimension three, one can directly compare the local models the fibrations studied by Gross in \cite{GrossBB, GrossSLagI} with the fibrations on Kato--Nakayama spaces studied in~\cite{AS}, to show that the Calabi--Yau compactifications studied in \cite{GrossTopology} are homeomorphic to $X_0^{\mathrm{KN}}(1)$, where $X_0$ is obtained from the \emph{trivial gluing data} \cite[Remark~$4.16$]{AS}. Note that the real locus in fibrations obtained from non-trivial gluing data typically have very different topologies, as studied in \cite[\S$5.3$]{AS}. Investigating a sequence, similar to \cite[Theorem~$1$]{CBM}, that also incorporates gluing data on $B$, to be able to capture the mod $2$ cohomology of the real loci in the case of non-trivial gluing data, is focus of future work.